\documentclass[reqno]{amsart}            
\usepackage{amsthm,amssymb,amsfonts,graphicx,cite,color}
\usepackage[bookmarks,bookmarksnumbered,bookmarksopen,colorlinks,backref,linkcolor=blue,citecolor=red]{hyperref}%
\usepackage{xcolor}
\usepackage{graphicx}
\usepackage{mathrsfs}

\def\R {\mathbb{R}}
\def\C {\mathcal{C}}
\def\S {\mathcal{S}}
\def\CC{{\mathbb C}}
\def\D{{\mathcal D}}
\def\P{\mathcal{P}}
\def\PP{{\boldsymbol{P}}}
\def\A{{\mathcal A}}
\def\FF{{\boldsymbol{F}}}

\newtheorem{theorem}{Theorem}[section]

\newtheorem{lemma}{Lemma}[section]
\newtheorem{corollary}{Corollary}[section]
\newtheorem{remark}{Remark}[section]
\newtheorem{proposition}{Proposition}[section]
\numberwithin{equation}{section}

\begin{document}
\title[ Asymptotic analysis of FBP]{Asymptotic analysis of a
contact  Hele-Shaw problem \\in a thin domain}

\author[T. Mel'nyk \  \& \  N. Vasylyeva]
{Taras Mel'nyk and  Nataliya Vasylyeva}

\address{Faculty of Mechanics and Mathematics, Taras Shevchenko National University of Kyiv,
\newline\indent
Volodymyrska st.\ 64, 01601 Kyiv, Ukraine}
\email[T.Mel'nyk]{melnyk@imath.kiev.ua}

\address{Institute of Applied Mathematics and Mechanics of NAS of Ukraine
\newline\indent
G.Batyuka st.\ 19, 84100 Sloviansk, Ukraine}
\email[N.Vasylyeva]{nataliy\underline{\ }v@yahoo.com}

\begin{abstract}
We analyze the  contact Hele-Shaw problem with zero surface tension of a free boundary in a thin domain $\Omega^{\varepsilon}(t).$ Under suitable conditions on the given data,  the one-valued local classical solvability of the problem for each fixed value of the parameter $\varepsilon$ is proved.

Using the multiscale analysis, we study the asymptotic behavior of this problem as $\varepsilon \to 0,$
i.e., when the thin domain $\Omega^{\varepsilon}(t)$  is shrunk into the interval $(0, l).$
Namely,  we find exact representation of the free boundary for $t\in[0,T],$
derive the corresponding limit problem $(\varepsilon= 0),$ define other terms of the asymptotic
approximation and prove appropriate asymptotic estimates that justify this approach.

 We also establish the preserving geometry of the free boundary near corner points for $t\in[0,T]$
 under assumption that free and fixed boundaries form  right angles at the initial time  $t=0$.
 \end{abstract}

\maketitle
\tableofcontents

\section{Introduction}
\label{s1}

\noindent

The Hele-Shaw problem was first introduced in 1897 by H.S. Hele-Shaw, a British engineer, scientist and inventor \cite{Hs1,Hs2}.
This problem models the pressure of fluid squeezed between two parallel plate, a small distance apart.
For the last 70 years, this problem have merited a great research interest among the mathematical, physical, engineering and biological community due to its  wide application in hydrodynamics, mathematical biology, chemistry and finance.
In addition, many other problems of fluid mechanics are associated with Hele-Shaw flows, and therefore the study of these flows is very important, especially for microflows. This is due to manufacturing technology that creates shallow flat configurations, and the typically low Reynolds numbers of microflows. There is a vast literature on the Hele-Shaw problem and related problems (see e.g. \cite{Ho}).

Here we focus on the contact one-phase Hele-Shaw problem with zero surface tension (ZST) of a free (unknown) boundary in a thin domain. Let $T>0$ be arbitrarily fixed, and let $Q\subset\R^{2}$ be a rectangle $Q=(0,l)\times(0,2\varepsilon)$ for some given positive values $l$ and $\varepsilon$.
We denote
$$
Q_{T}=Q\times (0,T)\quad\quad\text{and}\quad\quad \partial Q_{T}=\partial Q\times [0,T].
$$

Let $\Gamma^{\varepsilon}(t),$ $t\in[0,T],$ be a simple curve
$\Gamma^{\varepsilon}(t)\subset\bar{Q}$ which splits the rectangle
$Q$ into two subdomains $\Omega^{\varepsilon}(t)$ and $Q\backslash
\overline{\Omega^{\varepsilon}(t)}$, such that for some unknown function
$\rho=\rho(y_{1},t):[0,l]\times[0,T]\to\R$ the domain
$\Omega^{\varepsilon}(t)$ is given by
\begin{equation}\label{1.1}
\Omega^{\varepsilon}(t)=\{y=(y_{1},y_{2})\in Q:\quad
y_{1}\in(0,l),\quad 0<y_{2}<\varepsilon+\rho(y_{1},t)\},\quad
t\in(0,T)
\end{equation}
(see Fig. \ref{fig:1}).

The mathematical setting of the contact one-phase Hele-Shaw problem is to determine
the evolution of the 2-dimensional fluid domain
$\Omega^{\varepsilon}(t)$ (other words, to find a function $\rho$) and the fluid pressure
$p^{\varepsilon}=p^{\varepsilon}(y_{1},y_{2},t),$ $(y_{1},y_{2},t) \in \Omega^{\varepsilon}(t),$ such that
\begin{equation}\label{1.2}
\begin{cases}
\Delta_{y}p^{\varepsilon}=0\quad\text{in}\quad\Omega^{\varepsilon}(t),\quad t\in(0,T),
\\[2mm]
p^{\varepsilon}=0\quad\quad\text{and}\quad \dfrac{\partial p^{\varepsilon}}{\partial \mathbf{n}_{t}}=-\gamma V_{\mathbf{n}}\quad \text{on}\quad \Gamma^{\varepsilon}(t),\, t\in[0,T],
\\[2mm]
\dfrac{\partial p^{\varepsilon}}{\partial \mathbf{n}}=\Phi^{\varepsilon}(y,t)\quad \text{on}\quad\partial\Omega^{\varepsilon}(t)\backslash\Gamma^{\varepsilon}(t), \, t\in[0,T],
\\[2mm]
\rho(y_{1},0)=0,\qquad y_{1}\in[0,l],
\end{cases}
\end{equation}
where $\gamma$ is a positive given number and the function $\Phi^{\varepsilon}$
is prescribed, and $\mathbf{n}_{t}=(n_{t}^{1},n_{t}^{2})$ and $\mathbf{n}$ denote the outward normals to
$\Gamma^{\varepsilon}(t)$ and
$\partial\Omega^{\varepsilon}(t)\backslash\Gamma^{\varepsilon}(t)$,  respectively.
Finally, the symbol $V_{\mathbf{n}}$ stands the velocity
of the free boundary in the direction of $\mathbf{n}_{t}$ while $\Delta_{y}=\sum_{i=1}^{2}\frac{\partial^{2}}{\partial y_{i}^{2}}$.

 It is worth mentioning that the last condition in \eqref{1.2} together with representation \eqref{1.1} provides
 that the domain $\Omega^{\varepsilon}:=\Omega^{\varepsilon}(0)$  and, hence, $\Gamma^{\varepsilon}:=\Gamma^{\varepsilon}(0)$ are given. Moreover, the homogenous Dirichlet condition on the free boundary $\Gamma^{\varepsilon}(t)$ means that problem \eqref{1.1}-\eqref{1.2} is the Hele-Shaw  problem with ZST.

In the paper we analyze the well-posed problem \eqref{1.1}-\eqref{1.2}, that means
 the domain $\Omega^{\varepsilon} (t)$ is expanding in time $t\in[0,T]$, i.e. $\Omega^{\varepsilon} (t_{1})\subset \Omega^{\varepsilon} (t_{2})$ for $t_{1}<t_{2}$. This property can be achieved by the appropriate choice of the given function $\Phi^{\varepsilon}$.

\begin{figure}[htbp]
\begin{center}
\includegraphics[width=13cm]{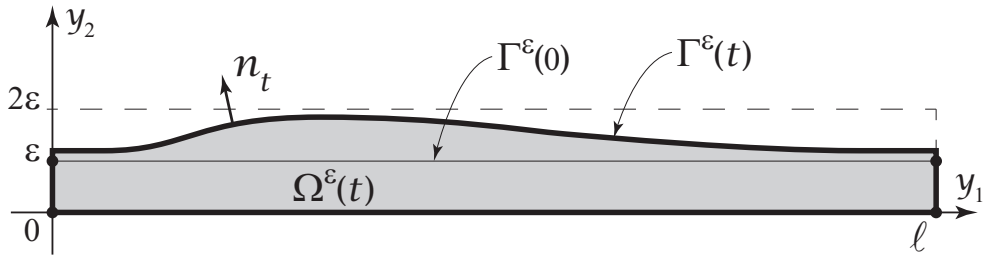}
\end{center}
\vskip-10pt
\caption{Typical domain configuration}\label{fig:1}
\end{figure}

Since the 1940s, there has been plenteous effort devoted to  the Hele-Shaw free boundary problem
(see e.g. \cite{Ga,Hy1,Hy2,HH,PK1,PK2,PK3,R,S,ST} and references therein) both analytically and numerically. The
significant steps leading to exact solutions of  Hele-Shaw models arose via conformal mapping techniques, by which the problem can be recast as an initial value problem of a functional differential equation.
Also if  a free boundary is spherically symmetric, there exists a unique radially symmetric stationary solution to a moving boundary problem \cite{Friedman_1999}. Stability and long-time behavior of solutions of the Hele-Shaw problem  have been extensively studied with different methods  in  \cite{C,CP,ES2,FR} (see also references therein). For further acquaintance with results,  we send readers to paper \cite{Va} and monograph \cite{GV}, where a brief overview of the Hele-Shaw problem and a historic overview of the development in
searching exact solutions is presented. As for numerical solutions of  Hele-Shaw flows, they were discussed in \cite{DM,CHS,W} (see also references therein).

Coming to the solvability of Hele-Shaw models, we quote the works \cite{EO,EJ,BF,Ki}, where existence of weak,
variational and viscosity solutions are established. In the case
of regular initial data, the existence and uniqueness of classical
solutions to the one-phase well-posed Hele-Shaw problem are discussed  in
\cite{B1,B,C,CP,ES,FR,P,Y}.

As for a nonregular initial shape of a moving boundary, the one-phase Hele-Shaw model with ZST was first investigated via qualitative approach in plain corners in \cite{KLV}, where the motion of the corner point was described. In
particular,  it was shown that the waiting time phenomena  (preservation of angles at moving boundaries for a certain time) exists in the case of acute angles, while the obstacle angles
are immediately smoothed. The solvability of  one- and two-phase
well-posed Hele-Shaw problems in the case of  corner points with
acute angles on a free boundary  were studied  in \cite{BF1,BV1,V1,V2}. We remark that papers \cite{BF1} and \cite{V2} are related to the contact  Hele-Shaw problem with and without surface tension of unknown boundary.

Nevertheless, the classical solvability of  problem \eqref{1.1}-\eqref{1.2} with nonhomogeneous Neumann conditions on the vertical sides $\Gamma^{\varepsilon}_{1}$ and $\Gamma^{\varepsilon}_{3}$  when fixed and free
boundaries form right angles is still an open problem.

 The motivation in the asymptotic study of problem \eqref{1.1}-\eqref{1.2} in the thin domain
$\Omega^{\varepsilon}(t)$ (as  $\varepsilon \to 0)$ arises from the investigation of mathematical models of atherosclerosis
\cite{FHH,Mel}.
In \cite{FHH} it was proved that for any small $\epsilon >0$ under certain condition for initial data  there is a unique $\epsilon$-thin radially symmetric stationary plaque, i.e., a plaque with $R(t)  \equiv 1 - \epsilon,$ and in addition, conditions were determined under which the $\epsilon$-thin stationary plaque is linearly asymptotically stable (or unstable) as $t \to +\infty$ and when it is shrunk and disappeared. A multiscale analysis of a new mathematical model of the atherosclerosis development  in a thin tubular domain (without moving boundary) was performed, in particular, the corresponding limit two-dimensional problem was derived, and the asymptotic approximation for the solution was constructed and justified in \cite{Mel}.

\smallskip

The novelty of this paper consists of three parts.
\begin{itemize}
  \item
  At first,  we find sufficiently conditions for given functions in model \eqref{1.1}-\eqref{1.2} which provide the local classical solvability of the contact Hele-Shaw problem for each fixed $\varepsilon>0.$ To this end, we exploit the approach from \cite{B1,V2}.
	  \item
Secondly, we make rigorous asymptotic analysis of problem \eqref{1.1}-\eqref{1.2} as  $\varepsilon \to 0,$ i.e., when the thin domain
$\Omega^{\varepsilon}(t)$  is shrunk into the interval $(0, l).$
Applying the method of papers \cite{KM,Mel,Mel_Sad}, we find the moving curve $\Gamma^{\varepsilon}(t)$ and
construct the asymptotic approximation  $\mathcal{P}^{\varepsilon}$ for the solution to problem \eqref{1.1}-\eqref{1.2} and evaluate its deviation from the classical solution $p^\varepsilon$ in the Sobolev space $\C([0,T]; \, H^{1}(\Omega^{\varepsilon}(t))).$
To our knowledge, this is the first work  in the mathematical literature concerning the rigorous asymptotic study of Stefan type  problems in thin domains.
\item
Finally, collecting the results concerning to both the solvability and the asymptotic representation, we establish preserving  the geometry of the free boundary in  small neighborhoods of the corner points. Besides, under certain assumptions on the given function $\Phi^{\varepsilon},$  the size of these neighborhoods is estimated via the size of the support of the function $\Phi^{\varepsilon}|_{y_{2}=0}$.
\end{itemize}

\vspace{5pt}

\noindent
{\bf Outline of the paper.}
 In Section \ref{s2}, we introduce some functional spaces and  notations.
The classical solvability of \eqref{1.1}-\eqref{1.2} is formulated in Theorem \ref{t3.1} in Section \ref{s3}.
Section \ref{s4}  states  the main Theorem \ref{t4.1} that describes the asymptotic behavior of the solution $p^\varepsilon$. Section \ref{s8} is devoted to obtaining some auxiliary results which play significant role in the proofs of Theorems \ref{t3.1} and \ref{t4.1}.
The proof of Theorem \ref{t3.1} is carried out in Section~\ref{s5}. Moreover, in Subsection~\ref{s5.6} we discuss the solvability problem \eqref{1.1}-\eqref{1.2} for  more general domains $\Omega^{\varepsilon}$ (i.e. $\rho(y_{1},0)\neq 0$; see Theorem \ref{t5.7}).  The proof of Theorem \ref{t4.1} is presented  in Section \ref{s6}. In Conclusion we analyze obtained results and consider research perspectives.


\section{Functional Spaces and Notations}
\label{s2}

We carry out our analysis in the framework of the H\"{o}lder and Sobolev spaces.
Therefore, we recall some definitions.
Let $\D$ be a domain in $\R^{n},$ $n\geq 1,$ and $\alpha\in(0,1).$ Notation
$\C^{k+\alpha}(\D),$ $L^{p}(\D),$ $W^{k,p}(\D),$ $W_{0}^{k,p}(\D)$ represent the classical  H\"{o}lder and Sobolev spaces, where
$k\in \mathbb{N}_0$ and $p\geq 1.$ In addition, we will use the standard alternative notation $H^{1}(\D)$  for the space $W^{1,2}(\D)$.

Let $\mathbf{X}$ be a  Banach space with the norm $\|\cdot\|_{\mathbf{X}}.$ The space $\C([0,T];\mathbf{X})$ comprises
all continuous function on $[0, T]$ taking values in $\mathbf{X};$
 the space $L^{p}((0,T);\mathbf{X})$ consists of all measurable functions $u \mapsto \mathbf{X}$
 with
 $$
 \|u\|_{L^{p}((0,T);\mathbf{X})} := \left(\int_0^T \|u(t)\|_{\mathbf{X}}^p\, dt\right)^{1/p} < +\infty.
 $$

Denote by $\D_{T}:=\D\times(0,T)$,
\[
\langle v\rangle_{y,\D_{T}}^{(\alpha)} := \sup\Big\{ \frac{|v(y,t)-v(\bar{y},t)|}{|y-\bar{y}|^{\alpha}}:\  (y,t), (\bar{y},t)\in\bar{\D}_{T}
\quad y\neq\bar{y}
\Big\},
\]
and
\[
\CC^{k+\alpha}(\bar{\D}_{T}) :=\C([0,T]; \C^{k+\alpha}(\bar{\D})).
\]

Also we introduce the Banach space $
\hat{\CC}^{k+\alpha}(\bar{\D}_{T}),\, k\geq 1, $ consisting of all
functions $v\in \CC^{k+\alpha}(\bar{\D}_{T})$ such that
$v_{t}:=\frac{\partial v}{\partial t}\in
\CC^{k-1+\alpha}(\bar{\D}_{T})$ and the norm
$$
\|v\|_{\hat{\CC}^{k+\alpha}(\bar{\D}_{T})}:=\|v\|_{\CC^{k+\alpha}(\bar{\D}_{T})}+
\|v_{t}\|_{\CC^{k-1+\alpha}(\bar{\D}_{T})} < +\infty.
$$

In the spaces $\CC^{k+\alpha}(\bar{\D}_{T})$ and $\hat{\CC}^{k+\alpha}(\bar{\D}_{T})$ we secrete the  subspaces
\begin{align*}
\CC_{0}^{k+\alpha}(\bar{\D}_{T})&:=\left\{v\in
\CC^{k+\alpha}(\bar{\D}_{T}) : \, D_{y}^{\beta}v(y,0)=0,\,
|\beta|=0,1,...,k\right\},
\\
\hat{\CC}_{0}^{k+\alpha}(\bar{\D}_{T})&:=\left\{v\in
\hat{\CC}^{k+\alpha}(\bar{\D}_{T}):  \,  D_{y}^{\beta}v(y,0)=0,\,
|\beta|=0,1,...,k,\, D_{y}^{\iota} v_{t}(y,0)=0,\,
|\iota|=0,1,...,k-1\right\},
\end{align*}
where $|\beta|$ and $|\iota|$ are  multyindexes, i.e.
$|\beta|=\beta_{1}+...+\beta_{n}$,
$|\iota|=\iota_{1}+...+\iota_{n}$ .

Throughout this work, the symbol $C$ will denote a generic positive constant, depending only on the structural quantities of the model.
We will denote the inner product in $L^{2}(0,a)$ by the symbol $\langle\cdot,\cdot\rangle_{a}$.

Finally, for each $t\in[0,T]$
the middle value of a function $v=v(z,t), \ z\in[0,\mathfrak{f}],$ is  designated by
\begin{equation}\label{middle_value}
\langle\langle v\rangle\rangle_{\mathfrak{f}}:=\frac{1}{\mathfrak{f}}\int_{0}^{\mathfrak{f}}v(z,t)dz,
\end{equation}
where $\mathfrak{f}$ is a positive function $\mathfrak{f}=\mathfrak{f}(\cdot,t)$.


\section{Local Classical Solvability of Problem \eqref{1.1}-\eqref{1.2}}
\label{s3}

\noindent Throughout this section, we assume that the positive
parameter $\varepsilon$ is arbitrary but fixed. First we write
$\Gamma^{\varepsilon}(t)$ and
$\partial\Omega^{\varepsilon}(t)\backslash\Gamma^{\varepsilon}(t)$
in  more comfortable form. Since we will look for the local
classical solution, we define the unknown boundary
$\Gamma^{\varepsilon}(t)$  for each $t\in[0,T]$ as follows
\begin{equation}\label{3.1}
\Gamma^{\varepsilon}(t)=\{y=(y_{1},y_{2})\in\R^{2}:\quad
y_{2}=\varepsilon+\rho(y_{1},t),\quad
y_{1}\in[0,l]\},\quad\text{where}\quad
|\rho(y_{1},t)|<\varepsilon/5.
\end{equation}
In the light of this definition, the boundary
$\partial\Omega^{\varepsilon}(t)\backslash\Gamma^{\varepsilon}(t)$
is described  for each $t\in[0,T]$ as
\[
\partial\Omega^{\varepsilon}(t)\backslash\Gamma^{\varepsilon}(t)=\Gamma^{\varepsilon}_{1}(t)\cup\Gamma_{2}\cup\Gamma^{\varepsilon}_{3}(t),
\]
where 
\begin{align}\label{3.2}\notag
\Gamma^{\varepsilon}_{1}(t)&=\{y=(y_{1},y_{2})\in\R^{2}:\quad y_{1}=0,\quad y_{2}\in[0,\varepsilon+\rho(0,t))\},\\
\Gamma_{2}&=\{y=(y_{1},y_{2})\in\R^{2}:\quad y_{2}=0,\quad y_{1}\in(0,l)\},\\
\Gamma^{\varepsilon}_{3}(t)&=\{y=(y_{1},y_{2})\in\R^{2}:\quad
y_{1}=l,\quad y_{2}\in[0,\varepsilon+\rho(l,t))\}.\notag
\end{align}

Now, we are ready to state our general  assumptions for  the structural quantities appearing in problem \eqref{1.1}-\eqref{1.2}.
\begin{description}
\item[(h1)(Conditions for the boundary
$\partial\Omega^{\varepsilon}(0)$)] We assume that
$$
\partial\Omega^{\varepsilon}:=\partial\Omega^{\varepsilon}(0)=\Gamma^{\varepsilon}_{1}\cup\Gamma_{2}\cup\Gamma^{\varepsilon}_{3}\cup\Gamma^{\varepsilon},
$$
where
\begin{align*}
\Gamma^{\varepsilon}_{1}&:=\Gamma^{\varepsilon}_{1}(0)=\{y=(y_{1},y_{2})\in\R^{2}:\quad y_{1}=0,\quad y_{2}\in[0,\varepsilon)\},\\
\Gamma^{\varepsilon}_{3}&:=\Gamma^{\varepsilon}_{3}(0)=\{y=(y_{1},y_{2})\in\R^{2}:\quad y_{1}=l,\quad y_{2}\in[0,\varepsilon)\},\\
\Gamma^{\varepsilon}&:=\Gamma^{\varepsilon}(0)=\{y=(y_{1},y_{2})\in\R^{2}:\quad
y_{2}=\varepsilon,\quad y_{1}\in[0,l]\}.
\end{align*}
Besides, for each fixed $T>0$, we denote
$\partial\Omega^{\varepsilon}_{ T}=\partial\Omega^{\varepsilon}\times [0,T],$ \ $\Gamma^{\varepsilon}_{ T}=\Gamma^{\varepsilon}\times [0,T],$ \ $\Gamma_{2,T}=\Gamma_{2}\times [0,T],$ \ $\Gamma^{\varepsilon}_{i,T}=\Gamma^{\varepsilon}_{i}\times [0,T]$ for  $i\in \{1,3\}.$

\item[(h2)(Smoothness of the given functions)]
Let
$$
\varphi_{1} \in\C([0,T];\C^{2+\alpha}[0,1]),\quad
\varphi_{2} \in\C([0,T];\C^{2+\alpha}[0,l]),\quad
\varphi_{3} \in\C([0,T];\C^{2+\alpha}[0,1]).
$$

\item[(h3)(Representation of the given function)]
We assume that
\begin{equation}\label{Phi}
\Phi^\varepsilon(y_{1},y_{2},t)=
\left\{
  \begin{array}{ll}
    \chi_{2}(\frac{y_{2}}{\varepsilon} )\, \varphi_{1}(\frac{y_{2}}{\varepsilon},t), & y_2 \in \Gamma^{\varepsilon}_{1}(t), \ \ t\in[0,T], \\[2mm]
    \varepsilon\, \chi_{1}(y_{1})\, \varphi_{2}(y_{1},t), &  y_1 \in \Gamma_{2}, \ \ t\in[0,T],\\[2mm]
    \chi_{2}(\frac{y_{2}}{\varepsilon}) \, \varphi_{3}(\frac{y_{2}}{\varepsilon},t), & y_2 \in \Gamma^{\varepsilon}_{3}(t), \ \ t\in[0,T],
  \end{array}
\right.
\end{equation}
where  $\chi_{i}\in\C_{0}^{\infty}(\R^{1}),$ $i\in \{1,2\},$  are the cut-off functions such that  $ 0\leq\chi_{i}\leq 1$ and
\begin{equation*}
  \chi_{1}(y_{1})=
\begin{cases}
1,\quad\text{if}\quad y_{1}\in[\frac{2l}{5} ,\frac{3l}{5}],\\[2mm]
0,\quad\text{if}\quad y_{1}\notin(\frac{l}{5} , \frac{4l}{5}),
\end{cases}
\quad
\chi_{2}(\xi_{2})=
\begin{cases}
1,\quad\text{if}\quad \xi_{2}\in[\frac{2}{5} , \frac{3}{5}],\\[2mm]
0,\quad\text{if}\quad \xi_{2}\notin( \frac{1}{5} , \frac{4}{5}).
\end{cases}
\end{equation*}
\item[(h4)(Condition of the well-posedness to \eqref{1.1}-\eqref{1.2})]
We require  that  the inequality holds
\begin{equation}\label{3.4}
V_{n}\Big|_{t=0}>0\quad \text{on}\quad \Gamma^{\varepsilon}.
\end{equation}
\end{description}
\begin{remark}\label{r3.1}
It is apparent that condition \eqref{3.4} means the positivity of the initial speed of the moving
boundary.  That guarantees the expansion of the domains
$\Omega^{\varepsilon}(t)$ ($\Omega^{\varepsilon}(t_{1})\subset
\Omega^{\varepsilon}(t_{2})$ if $0\leq t_{1} < t_2 \, \leq T$) and as a consequence the well-posedness of
\eqref{1.1}-\eqref{1.2} (see, e.g. \cite{BF}, \cite{Va}).
Besides, this speed, obviously, depends on the function $\Phi^{\varepsilon}$.
In  forthcoming Lemma \ref{l5.1} (Subsection \ref{s5.1}) and Remark \ref{r4.3} (Subsection \ref{s4.1}), we shall discuss the assumptions on  $\Phi^{\varepsilon}$ which provide inequality \eqref{3.4}. In addition, the formal  integration by parts in \eqref{1.2} gives the necessary condition
\[
\int_{0}^{1}\chi_{2}(\xi_{2})\varphi_{1}(\xi_{2},t) \, d\xi_{2}+\int_{0}^{l}\chi_{1}(y_{1})\varphi_{2}(y_{1},t) \, dy_{1}+\int_{0}^{1}\chi_{2}(\xi_{2},t)\varphi_{3}(\xi_{2},t) \, d\xi_{2}>0\qquad \forall\, t\in[0,T]
\]
for the fulfillment of \eqref{3.4}.
\end{remark}

Now we can state our first main result concerning the local classical solvability of the Hele-Shaw problem \eqref{1.1}-\eqref{1.2}.
\begin{theorem}\label{t3.1}
Under assumptions \textbf{(h1)}-\textbf{(h4)}, for any fixed positive $\varepsilon$,
problem \eqref{1.1}-\eqref{1.2} admits a unique classical solution
$(p^{\varepsilon}(y_{1},y_{2},t),\rho(y_{1},t))$ in some
interval $t\in[0,T]$, such that $\Gamma^{\varepsilon}(t)$ is given
by \eqref{3.1} and
\[
p^{\varepsilon}\in\C([0,T];\C^{2+\alpha}(\overline{\Omega^{\varepsilon}(t)})),\quad
\rho\in\C([0,T];\C^{2+\alpha}([0,l])), \quad
\frac{\partial\rho}{\partial t}\in\C([0,T];\C^{1+\alpha}([0,l])).
\]
\end{theorem}

\begin{remark}\label{r3.3}
To simplicity consideration, we specify  supports of the cut-off functions $\chi_{1}$ and $\chi_{2}$ in assumption \textbf{(h3)}. Actually, with the nonessential modifications in the proof, the same results hold for the functions $\chi_{1}$ and $\chi_{2}$ with the supports lying strictly inside $(0,l)$ and $(0,1),$ respectively.
\end{remark}

\begin{remark}\label{r3.2} One-valued classical solvability of \eqref{1.1}-\eqref{1.2} can be provided by more general assumptions on
$\partial\Omega^{\varepsilon}$ and $\Phi^{\varepsilon}$. This will be discussed in Theorem \ref{t5.7} (see Subsection \ref{s5.6}).
\end{remark}

The proof of Theorem \ref{t3.1}, which is rather technical, will be postponed to Section \ref{s5}.


\section{Asymptotic behavior of the classical solution to \eqref{1.1}-\eqref{1.2}}
\label{s4}

\noindent
Our next results are concerned the asymptotic behavior of the solution $p^{\varepsilon}$ as $\varepsilon\to 0$.
First, being within  assumptions of Theorem \ref{t3.1}, we  rewrite problem \eqref{1.1}-\eqref{1.2} in more convenient form for studying its asymptotic behavior. To this end,
 we represent $\Gamma^{\varepsilon}(t)$ as
\begin{equation}\label{4.00}
\Gamma^{\varepsilon}(t)=\left\{y= (y_{1},y_{2})\in\R^{2}:\quad
y_{1}\in[0,l],\quad y_{2}=\varepsilon\S(y_{1},t)\right\} \quad \forall \, t\in[0,T],
\end{equation}
where  $\S(y_{1},t)=1+\varepsilon^{-1}\rho(y_{1},t),$ \ $y_{1}\in[0,l], \  t\in[0,T],$
is a new unknown function. Due to \eqref{3.1} and $\bf{(h1)}$
\begin{equation}\label{4.0}
\S(y_{1},0)=1\quad\text{for all} \quad y_{1}\in[0,l],\quad\text{and}\quad  |\S(y_{1},t)|<6/5.
\end{equation}

Below we present the important property of the function $\mathcal{S},$ which will  be used  in the asymptotic analysis
of our problem.
\begin{corollary}\label{c4.1}
Under conditions of Theorem \ref{t3.1}, for all $t\in [0,T]$, the unknown function
$\mathcal{S}$ admits representations
\[
\mathcal{S}(y_{1},t)=\begin{cases}
\mathcal{S}_{0}(t)\quad \text{in a}\   \ \delta\text{-neighborhood of}\quad y_1=0,\\
\mathcal{S}_{l}(t)\quad\text{in a }\ \delta\text{-neighborhood of}\quad y_1=l,
\end{cases}
\]
where $\delta$ is a small positive number. Besides,
\[
\mathcal{S}_{0}(t)\geq 1\quad \text{and}\quad \mathcal{S}_{l}(t)\geq 1\quad \forall t\in[0,T].
\]
\end{corollary}
The proof of this statement is a simple consequence of Theorem \ref{t3.1} and it bases on the homogenous Dirichlet condition on $\Gamma^{\varepsilon}(t)$ in problem \eqref{1.2} and on the property of the function $\Phi^{\varepsilon}$ (recall that assumption \textbf{(h3)} provides homogenous Neumann conditions  near the contact  points of  free and fixed  boundaries).
\begin{remark}\label{r4.0}
The statement of Corollary \ref{c4.1} means that the geometry of the free boundary in $\delta$-neighborhoods of the corner points $(0,\mathcal{S}_{0}(t),t),$ $(l,\mathcal{S}_{l}(t),t)$, preserves for each $t\in[0,T]$. In forthcoming  Corollary \ref{c4.2}, under some additional assumptions on the given data, we estimate the value $\delta$ through the support of the function $\Phi^{\varepsilon}|_{y_{2}=0}$ and obtain an  explicit formula for unknown functions $\mathcal{S}_{l}(t)$ and $\mathcal{S}_{0}(t)$.
\end{remark}

Obviously, that  for each $t\in[0,T]$ on the boundary $\Gamma^{\varepsilon}(t)$ we have
\begin{equation}\label{4.1}
\mathcal{R}(y,t):=y_{2}-\varepsilon \mathcal{S}(y_1,t)=0.
\end{equation}
This  equality with straightforward calculations arrive at the following relations on the free boundary:
\[
\frac{\partial p^{\varepsilon}}{\partial {\bf n}_{t}}= \sum_{i=1}^{2}\frac{\partial p^{\varepsilon}}{\partial y_{i}}n_{t}^{i}
=  \frac{1}{|\nabla_{y}\mathcal{R}|} \sum_{i=1}^{2}\frac{\partial p^{\varepsilon}}{\partial y_{i}}\frac{\partial \mathcal{R}}{\partial y_{i}}
\qquad\text{and}\quad
V_{{\bf n}}= \sum_{i=1}^{2} \frac{d y_i}{dt} \, n_{t}^{i} = \frac{\varepsilon }{|\nabla_{\bar{y}}\mathcal{R}|}
\, \frac{\partial \mathcal{S}}{\partial t},
 \]
where  $\nabla_{{y}}\mathcal{R}= (-\varepsilon \frac{\partial \mathcal{S}}{\partial y_1} , 1).$

Thus, the Stefan condition on the free boundary can  be rewritten as
\begin{equation}\label{4.2}
\frac{\partial p^{\varepsilon}}{\partial y_{2}}=\varepsilon
\frac{\partial \mathcal{S}}{\partial y_{1}}  \frac{\partial p^{\varepsilon}}{\partial
y_{1}}-\varepsilon \gamma \frac{\partial \mathcal{S}}{\partial t}.
\end{equation}
Since Theorem \ref{t3.1}  provides the one-to-one classical solvability of \eqref{1.1}-\eqref{1.2},
we can  integrate the first condition on the moving boundary along $\Gamma^{\varepsilon}(t)$ (here we keep in mind the line integral of the first kind). As a result, the
classical solution $(p^{\varepsilon},\mathcal{S})$ of
\eqref{1.1}-\eqref{1.2} satisfies the problem
\begin{equation}\label{4.3}
\begin{cases}
\Delta_{y}p^{\varepsilon}=0\quad \text{in}\quad \Omega^{\varepsilon}(t),\quad t\in(0,T),
\\\\
\displaystyle{
\frac{\partial p^{\varepsilon}}{\partial y_{2}}=\varepsilon
\frac{\partial \mathcal{S}}{\partial y_{1}}  \frac{\partial p^{\varepsilon}}{\partial
y_{1}}-\varepsilon \gamma \frac{\partial \mathcal{S}}{\partial t}} \quad
\text{on}\quad \Gamma^{\varepsilon}(t), \quad t\in[0,T],
\\\\
\displaystyle{\int_{\Gamma^{\varepsilon}(t)} p^{\varepsilon} \, d\ell} = 0,  \quad t\in[0,T],
\\\\
\displaystyle{ \frac{\partial p^{\varepsilon}}{\partial {\bf n}}=\Phi^\varepsilon(y,t)}\quad\text{on}\quad \partial\Omega^{\varepsilon}(t)\backslash\Gamma^{\varepsilon}(t),\quad t\in[0,T],
\\\\
\mathcal{S}(y_{1},0)=1\quad\text{on}\quad [0,l].
\end{cases}
\end{equation}
Usually, to construct the asymptotic approximation for a solution to a boundary-value problem, more smoothness is required for the initial data. In our case, they are
\begin{description}
  \item[\textbf{(h5)}] \qquad
$\varphi_{1} \in\C([0,T];\C^{3}[0,1]),\quad
  \varphi_{2} \in\C([0,T];\C^{3}[0,l]),\quad
\varphi_{3} \in\C([0,T];\C^{3}[0,1]).$
\end{description}

\begin{theorem}\label{t4.1}
Let assumptions \textbf{(h1), (h3)}, \textbf{(h4)}  and \textbf{(h5)} hold. Then there exist positive constants $C_{0}$ and $\varepsilon_{0}$, such that for each $\varepsilon\in(0,\varepsilon_{0})$ the free boundary $\Gamma^{\varepsilon}(t)$ is uniquely determined by means of the function
\begin{equation}\label{S}
\S(y_{1},t)=1+\frac{1}{\gamma}\int_{0}^{t}\chi_{1}(y_{1})\varphi_{2}(y_{1},\tau)d\tau+\frac{1}{l\gamma}\int_{0}^{t}d\tau
\int_{0}^{1}\chi_{2}(\xi_{2})[\varphi_{3}(\xi_{2},\tau)+\varphi_{1}(\xi_{2},\tau)]d\xi_{2}, \quad t\in[0,T],
\end{equation}
and the following inequality holds:
\begin{equation}\label{main_estimate}
  \|p^{\varepsilon}-\P^{\varepsilon}\|_{\C([0,T]; \, H^{1}(\Omega^{\varepsilon}(t)))}\leq
C_0 \, \varepsilon
\end{equation}
 with  $p^\varepsilon$ being  the classical solution to problem \eqref{1.1}-\eqref{1.2} provided by Theorem \ref{t3.1}, while the approximation function
 \begin{equation*}
\P^{\varepsilon}(y,t)=\mathfrak{w}_{0}(y_{1},t) + \varepsilon^{2}u_{2}(y_{1},y_{2}/\varepsilon,t), \quad y\in \Omega^{\varepsilon}(t), \ \ t\in [0,T].
\end{equation*}
Here, for each $t\in[0,T]$ the functions $\mathfrak{w}_{0}$ and $u_{2}$ are unique classical solutions to the problems
\begin{equation}\label{limit_prob}
\begin{cases}
\dfrac{\partial}{\partial y_1}\left( \mathcal{S}(y_1,t)  \dfrac{\partial \mathfrak{w}_0(y_1,t)}{\partial{y_1}}  \right) = \gamma \dfrac{\partial \mathcal{S}}{\partial t}(y_1,t) - \chi_{1}(y_{1})\, \varphi_{2}(y_{1},t), \quad y_1\in (0, l),
\\\\
\dfrac{\partial \mathfrak{w}_0}{\partial{y_1}} (0,t) =
- \langle\langle \chi_{2}( \cdot)\, \varphi_{1}(\cdot,t)\rangle\rangle_{\mathcal{S}_0},
\qquad
\dfrac{\partial \mathfrak{w}_0}{\partial{y_1}} (l,t) =
 \langle\langle \chi_{2}( \cdot)\, \varphi_{3}(\cdot,t)\rangle\rangle_{\mathcal{S}_l},
 \\\\
\displaystyle{\frac{1}{|\Gamma^{\varepsilon}(t)|}\int_{\Gamma^{\varepsilon}(t)}\mathfrak{w}_{0}\, d\ell = 0} ;
\end{cases}
\end{equation}

\begin{equation*}
\begin{cases}
- \dfrac{\partial^{2} u_2}{\partial \xi^2_2 }(y_1, {\xi}_2,t)
  =
\dfrac{\partial^{2} \mathfrak{w}_0}{\partial y_1^2} (y_1,t),
 \quad \xi_2 \in (0, \mathcal{S}(y_1,t)), \quad y_{1}\in(0,l),
\\\\
 \dfrac{\partial u_2}{\partial \xi_2}\big(y_1, \mathcal{S}(y_1,t), t\big)
  =  \dfrac{\partial \mathcal{S}}{\partial y_{1}} \, \dfrac{\partial \mathfrak{w}_0}{\partial y_{1}} - \gamma \dfrac{\partial \mathcal{S}}{\partial t}, \qquad u_{2}(y_{1},S(y_{1},t),t)=0,
	\quad y_{1}\in[0,l],
	\\\\
 \dfrac{\partial u_2}{\partial \xi_2}(y_1, 0,t)
  = - \chi_{1}(y_{1})\, \varphi_{2}(y_{1},t),
 \end{cases}
\end{equation*}
respectively, where $|\Gamma^{\varepsilon}(t)|$ denotes the length of the curve $\Gamma^{\varepsilon}(t)$  and
the value $\langle\langle \cdot \rangle\rangle_{\mathcal{S}}$ is determined in \eqref{middle_value}.
\end{theorem}
The proof of this statement is given in Section \ref{s6}.
Collecting Theorem \ref{t4.1},  Corollary \ref{c4.1} and Remark \ref{r4.0}, we get  the key property of the free boundary  $\Gamma^{\varepsilon}(t)$ (see \eqref{4.00}).
\begin{corollary}\label{c4.2}
Under condition of Theorem \ref{t4.1}, the free and fixed boundaries in problem \eqref{1.1}-\eqref{1.2} form right angles
 in  $\delta$-neighborhoods of the corner points $(0,\mathcal{S}_{0}(t),t),$ $(l,\mathcal{S}_{l}(t),t)$, respectively, for $\delta=\frac{l}{5}$, $t\in[0,T]$ and $\varepsilon\in (0,\varepsilon_{0})$.  Besides, it follows from \eqref{S} that
\[
\mathcal{S}_{0}(t)=\mathcal{S}_{l}(t)=1+\frac{1}{l\gamma}\int_{0}^{t}d\tau
\int_{0}^{1}\chi_{2}(\xi_{2})[\varphi_{3}(\xi_{2},\tau)+\varphi_{1}(\xi_{2},\tau)]d\xi_{2},\quad t\in[0,T].
\]
\end{corollary}


\section{Some additional results}
\label{s8}

First, we describe some properties (which will be very useful in our analysis in Section \ref{s5}) of the eigenvalues and eigenfunction to the following spectral problems.
$$
\begin{cases}
- \psi''(x) = \lambda \, \psi(x),\quad x\in(0,a),\\
\psi'(0)=\psi'(a)=0,
\end{cases}
\quad \text{and} \quad
\begin{cases}
- \psi''(x) = \mu \, \psi(x),\quad y\in(0,a).\\
\psi'(0)=\psi(a)=0.
\end{cases}
$$
Obviously  that eigenvalues  of the spectral problems
are equal to
\begin{equation}\label{2.0}
\lambda_{m}:=\lambda_{m}(a)=\Big(\frac{\pi m}{a}\Big)^{2}\quad\text{and}\quad \mu_{m}:=\mu_{m}(a)=\Big(\frac{\pi (m+1/2)}{a}\Big)^{2}\quad m \in \mathbb{N}_0,
\end{equation}
respectively, and the corresponding eigenfunctions
\begin{equation}\label{2.0*}
 \psi_{\lambda_{0}}= \frac{1}{\sqrt{a}}, \quad
    \psi_{\lambda_{m}}(x)= \sqrt{\frac{2}{a}} \cos\sqrt{\lambda_{m}}x\quad \text{if}\,\,\,  m\neq 0, \quad\text{and}\,\,
 \psi_{\mu_{m}}(x)=\sqrt{\frac{2}{a}}\cos\sqrt{\mu_{m}}x,\, m\in  \mathbb{N}_0,
\,\, x\in[0,a],
\end{equation}
satisfy the following relations:
$$
\|\psi_{\lambda_{m}}\|_{L_{2}(0,a)}=\|\psi_{\mu_{m}}\|_{L^{2}(0,a)}=1,\quad \langle\psi_{\lambda_{m}},\psi_{\lambda_{n}}\rangle_{a}=0,\quad
\langle\psi_{\mu_{m}},\psi_{\mu_{n}}\rangle_{a}=0\quad \text{if}\quad m\neq n.
$$

For any function $g\in L^{2}(0,a),$ we notate by
\begin{equation*}
g_{\lambda_{m}}:= \langle g,\psi_{\lambda_{m}}\rangle_{a}\quad\text{and}\quad g_{\mu_{m}}:= \langle g,\psi_{\mu_{m}}\rangle_{a},\quad m\in \mathbb{N}_0,
\end{equation*}
 Fourier coefficients regarding the basis $\{\psi_{\lambda_{m}}\}_{m\in \mathbb{N}_0}$ and  $\{\psi_{\mu_{m}}\}_{m\in \mathbb{N}_0},$
respectively.

The next property demonstrates the correlation between the smoothness of the function $g$ and the behavior of its Fourier coefficients $\{g_{\lambda_{m}}\}$ and $\{g_{\mu_{m}}\}$.
\begin{proposition}\label{p2.1}
Let $g\in\C^{2+\alpha}([0,a])$ and
\begin{equation}\label{2.1}
g(0)=g(a)=0\quad\text{and}\quad g'(0)=g'(a)=0.
\end{equation}
Then the series
$$
\sum_{m=1}^{+\infty}g_{\lambda_{m}} \psi_{\lambda_{m}}(y)\quad \text{and} \quad \sum_{m=0}^{+\infty}g_{\mu_{m}} \psi_{\mu_{m}}(y)
$$
absolutely and uniformly converge on  $[0,a]$.

Besides, for $k\in \{0,1,2\},$ the inequalities are fulfilled
\begin{equation}\label{2.2}
\sum_{m=1}^{\infty}|g_{\lambda_{m}}| \, (\lambda_{m})^{\frac{k-1}{2}}\leq C\|g\|_{\C^{2+\alpha}([0,a])},\quad
\sum_{m=0}^{\infty}|g_{\mu_{m}}| \, (\mu_{m})^{\frac{k-1}{2}}\leq C\|g\|_{\C^{2+\alpha}([0,a])}.
\end{equation}
In addition, if  $\alpha\in(\frac12, 1),$ then estimates \eqref{2.2} hold for $k=3$.
\end{proposition}
\begin{proof}
The first statement is a simple consequence of the Fourier series theory (see e.g. Chapters I-IV in \cite{Ba}). Next, we will carry out the detailed proof of inequality \eqref{2.2} in the coefficients $\{g_{\lambda_{m}}\}$. The proof for $\{g_{\mu_{m}}\}$ is the same.

Taking into account the smoothness of $g,$ conditions \eqref{2.1} and integrating twice by parts in the representation of $g_{\lambda_{m}}$, we conclude
$$
|g_{\lambda_{m}}| \leq C\Big|\int_{0}^{a}\frac{g''(y)\cos\sqrt{\lambda_{m}}y}{\lambda_{m}}dy\Big|
=\frac{C g''_{\lambda_{m}}}{\lambda_{m}}
\leq \frac{C \|g''\|_{\C([0,a])}}{\lambda_{m}}, \quad m\in \mathbb{N},
$$
where $g''_{\lambda_{m}}=\langle g'',\psi_{\lambda_{m}}\rangle_{a}$.
These inequalities arrive at the estimate
\[
\sum_{m=1}^{+\infty}|g_{\lambda_{m}}|(\lambda_{m})^{\frac{k-1}{2}}\leq \sum_{m=1}^{+\infty}|g''_{\lambda_{m}}|(\lambda_{m})^{\frac{k-3}{2}}
\]
for $k\in \{0,1,2,3\}.$
Then,  the straightforward calculations provide the inequality
\[
\sum_{m=1}^{+\infty}|g''_{\lambda_{m}}|(\lambda_{m})^{\frac{k-3}{2}}\leq C\|g''\|_{\C([0,a])}\sum_{m=1}^{+\infty}m^{k-3}\leq C\|g\|_{\C^{2}([0,a])}.
\]
for $k\in \{0,1\},$ and as a result, estimate \eqref{2.2} for those values of $k$.

From \cite[Ch. 2]{Ba}  it follows the inequalities
$$
\sum_{m=1}^{+\infty}|g''_{\lambda_{m}}|(\lambda_{m})^{\frac{k-3}{2}}\leq
C
\begin{cases}
\sqrt{\sum\limits_{m=1}^{+\infty}|g''_{\lambda_{m}}|^{2}}\sqrt{\sum\limits_{m=1}^{+\infty}m^{-2}}\quad\text{if}\quad k=2,\\
\sum\limits_{m=1}^{+\infty}|g''_{\lambda_{m}}| \quad\text{if}\quad k=3,
\end{cases}
$$
whence
$$
\sum_{m=1}^{\infty}|g''_{\lambda_{m}}|(\lambda_{m})^{\frac{k-3}{2}}\leq
C\|g\|_{\C^{2+\alpha}([0,a])}.
$$
It is worth noting that the last estimate is true for $k=3$ if $\alpha\in(1/2,1)$ (see Corollary 2  \cite[Ch.~ 2]{Ba}).
\end{proof}

Let us consider the smooth cut-off  function $\chi \in\C_{0}^{\infty}(\R)$ such that $0\le \chi \le 1$ and
\[
\chi(y)=
\begin{cases}
0,\quad y\in(-\infty, a/5]\cup[4a/5,+\infty),\\
1,\quad y\in[2a/5,3a/5].
\end{cases}
\]

\begin{remark}\label{r2.1} Clearly that the function
$g =\chi  \, g_{0},$ where $g_{0}\in\C^{2+\alpha}([0,a]),$ meets all the requirements of Proposition \ref{p2.1}.
\end{remark}



We conclude this preliminary section with an analogue of the Poincar\'e inequality which will play a key point in the proof of Theorem \ref{t4.1} (see Section \ref{s4.2}).

\begin{lemma}\label{l4.5}
Let  the domain $\Omega^{\varepsilon}(t)=\left\{y= (y_{1},y_{2})\in\R^{2}:\,
y_{1}\in(0,l),\, 0<y_{2}<\varepsilon\S(y_{1},t)\right\}$ have a Lipschitz boundary for each $t\in[0,T]$, and let the given function $\S\in \C([0,T];\C^{1}([0,l]))$
define the curve $\Gamma^{\varepsilon}(t)$ according to formula  \eqref{4.00}.
Then
there is a constant $C$ such that for all $\varepsilon \in (0,1)$ and
  $t\in[0,T]$ the inequality
\begin{equation}\label{Puankare}
  \|\mathfrak{Y}\|_{L^2(\Omega^\varepsilon(t))} \le C \, \|\nabla_y \mathfrak{Y} \|_{L^2(\Omega^\varepsilon(t))}
\end{equation}
 holds for every function $\mathfrak{Y} \in H^1(\Omega^\varepsilon(t))$ such that
 $\int_{\Gamma^{\varepsilon}(t)} \mathfrak{Y} \, d\ell = 0.$
\end{lemma}
\begin{proof} We fixate $t\in [0,T]$ and  prove first this lemma for  a smooth function  $\mathfrak{Y} \in C^1\big(\overline{\Omega^\varepsilon(t)}\big).$ Taking into advantage of the easily verified identity
\begin{equation*}
\mathfrak{Y}(y_1, \varepsilon \mathcal{S}(y_1,t)) = \int_{y_2}^{\varepsilon \mathcal{S}(y_1,t)} \frac{\partial \mathfrak{Y}}{\partial z}(y_1, z)\, dz +  \mathfrak{Y}(y_1,y_2)
\quad \forall\, (y_1, y_2) \in \Omega^\varepsilon(t),
\end{equation*}
we deduce the inequalities
\begin{equation}\label{l_2}
\int_{\Omega^\varepsilon(t)} \mathfrak{Y}^2\, dy  \le C \left( \varepsilon^2 \int_{\Omega^\varepsilon(t)} \Big(\frac{\partial \mathfrak{Y}}{\partial y_2}\Big)^2 dy +
  \varepsilon \int_{\Gamma^{\varepsilon}(t)} \mathfrak{Y}^2  \, d\ell \right),
\end{equation}

\begin{multline}\label{l_3}
\mathfrak{Y}^2(y'_1, \varepsilon \mathcal{S}(y'_1,t))  - 2\, \mathfrak{Y}(y'_1, \varepsilon \mathcal{S}(y'_1,t)) \, \mathfrak{Y}(y_1, \varepsilon \mathcal{S}(y_1,t))  + \mathfrak{Y}^2(y_1, \varepsilon \mathcal{S}(y_1,t))
  \\
  \le C \left(\varepsilon \int_{0}^{\varepsilon \mathcal{S}(y'_1,t)} \Big(\frac{\partial \mathfrak{Y}}{\partial z_{2}}(y'_1, z_{2})\Big)^2  dz_{2} + \varepsilon \int_{0}^{\varepsilon \mathcal{S}(y_1,t)} \Big(\frac{\partial \mathfrak{Y}}{\partial z_{2}}(y_1, z_{2})\Big)^2  dz_{2} +
  \int_{0}^{l} \Big(\frac{\partial \mathfrak{Y}}{\partial z_{1}}(z_{1}, y_2)\Big)^2  dz_{1}\right)
\end{multline}
for each $y_1$ and $y'_1$ from the interval $(0, l)$ and any $y_2 \in (0, \varepsilon).$

After that  we integrate  inequality \eqref{l_3} with respect to $y_2 \in (0, \varepsilon).$ Then, we multiply the obtained inequality with
$\sqrt{1 + \varepsilon^2 \big(\frac{\partial \mathcal{S}}{\partial y_1}(y_1,t)\big)^2}$ and integrate it with respect to $y_1 \in (0, l).$
Finally,  multiplying the newly obtained inequality with  $\sqrt{1 + \varepsilon^2 \big(\frac{\partial \mathcal{S}}{\partial y'_1}(y'_1,t)\big)^2}$ and integrating with respect to $y_{1}'\in(0,l)$, we reach the estimate
\begin{align}\label{l_4}\notag
 \varepsilon \int_{\Gamma^{\varepsilon}(t)} \mathfrak{Y}^2  \, d\ell &\le
 \varepsilon \left(\int_{\Gamma^{\varepsilon}(t)} \mathfrak{Y} \, d\ell\right)^2 +
 C \left(\varepsilon^2  \int_{\Omega^\varepsilon(t)} \Big(\frac{\partial \mathfrak{Y}}{\partial y_2}\Big)^2 dy + \int_{\Omega^\varepsilon(t)} \Big(\frac{\partial \mathfrak{Y}}{\partial y_1}\Big)^2\, dy
   \right)\\
	&
	\le \varepsilon \left(\int_{\Gamma^{\varepsilon}(t)} \mathfrak{Y} \, d\ell\right)^2 +
 C \int_{\Omega^\varepsilon(t)} |\nabla_y \mathfrak{Y}|^2 dy.
\end{align}
Exploiting standard  approximation procedure, we conclude that  inequalities \eqref{l_2} and  \eqref{l_4}  hold for any function $\mathfrak{Y} \in H^1(\Omega^\varepsilon(t)).$
Since $\int_{\Gamma^{\varepsilon}(t)} \mathfrak{Y} \, d\ell = 0,$ estimate \eqref{Puankare} follows
from \eqref{l_2} and  \eqref{l_4}.
\end{proof}

\begin{corollary}
Estimate \eqref{l_4} immediately leads to an analogue of the Poincar\'e--Wirtinger inequality
\begin{equation}\label{Poi_Wirtinger}
\left\|\mathfrak{Y} - \frac{1}{|\Gamma^{\varepsilon}(t)|}\int_{\Gamma^{\varepsilon}(t)}\mathfrak{Y} \, d\ell\right\|_{L^2(\Gamma^\varepsilon(t))} \le \frac{C}{\sqrt{\varepsilon}} \, \|\nabla_y \mathfrak{Y}\|_{L^2(\Omega^\varepsilon(t))}
\quad \forall\, \mathfrak{Y} \in H^1(\Omega^\varepsilon(t)),
\end{equation}
where the constant $C$ is independent of $\mathfrak{Y}$ and $\varepsilon.$
\end{corollary}

\begin{remark}
  It should be noted here that determining the optimal constant in Poincar\'e inequalities is, in general, a very hard task (see \cite{Aco-Dur,Pay-Wei}). In our case it is very important to know how constants in such inequalities depend on the parameter $\varepsilon$ $($see \eqref{Puankare} and \eqref{Poi_Wirtinger}$)$.
\end{remark}


\section{Proof of Theorem \ref{t3.1}}
\label{s5}

\noindent
The strategy of the proof is the following: first, we show that, within our assumptions on the function $\Phi^{\varepsilon}$,
 the initial pressure $p^{\varepsilon}_{0}=p^{\varepsilon}(y,0)$ belongs to the class $\C^{3+\alpha}(\bar{\Omega}^{\varepsilon})$
 for any fixed $\varepsilon>0$. Then, using like Hanzawa transformation \cite{Ha}, we reduce problem \eqref{1.2} in the domain \eqref{1.1}
 with the moving boundary $\Gamma^{\varepsilon}(t)$ to a nonlinear problem in the fixed domain $\Omega^{\varepsilon}_{ T}$.
  After that we linearize this nonlinear problem on the initial data $p^{\varepsilon}_{0}$ and on a special function $s(x_{1},t)$
  connected with the initial shape of the free boundary $\Gamma^{\varepsilon},$ and solve the linear problem in $\CC^{2+\alpha}(\bar{\Omega}^{\varepsilon}_{T})$.
   Finally, using contraction mapping theorem, we prove the local one-to-one solvability of the corresponding nonlinear problem.


\subsection{Smoothness of the initial pressure $p^{\varepsilon}_{0}$}
\label{s5.1}

\noindent
Denoting (see assumption \textbf{(h3)})
\[
\bar{\varphi}_{1}(y_{2},t)=-\chi_{2}(y_{2}/\varepsilon)\varphi_{1}(y_{2}/\varepsilon,t),\quad
\bar{\varphi}_{2}(y_{1},t)=-\varepsilon\chi_{1}(y_{1})\varphi_{2}(y_{1},t),\quad
\bar{\varphi}_{3}(y_{2},t)=\chi_{2}(y_{2}/\varepsilon)\varphi_{3}(y_{2}/\varepsilon,t),
\]
and taking into account assumptions \textbf{(h1)-(h4)}, we conclude that
the initial pressure
$p^{\varepsilon}_{0}:\Omega^{\varepsilon}\to\R$ solves the
following boundary-value problem for each fixed $\varepsilon>0$:
\begin{equation}\label{5.1}
\begin{cases}
\Delta p^{\varepsilon}_{0}=0\quad\text{in}\quad \Omega^{\varepsilon},\\
p^{\varepsilon}_{0}=0\quad\text{on}\quad \Gamma^{\varepsilon},\\
\dfrac{\partial p^{\varepsilon}_{0}}{\partial y_{1}}=\bar{\varphi}_{1}(y_{2},0)\quad\text{on}\quad \Gamma^{\varepsilon}_{1},\\
\dfrac{\partial p^{\varepsilon}_{0}}{\partial y_{2}}=\bar{\varphi}_{2}(y_{1},0)\quad\text{on}\quad \Gamma_{2},\\
\dfrac{\partial p^{\varepsilon}_{0}}{\partial
y_{1}}=\bar{\varphi}_{3}(y_{2},0)\quad\text{on}\quad
\Gamma^{\varepsilon}_{3}.
\end{cases}
\end{equation}

Introducing new functions
\begin{align*}
\PP_{0}&=\varepsilon^{-1/2}(y_{2}-\varepsilon)\bar{\varphi}_{2,0}(0),\quad
\PP_{1}=\sum_{m=1}^{\infty}\frac{\bar{\varphi}_{2,m}(0)}{\sqrt{\lambda_{m}}}\frac{\sinh((y_{2}-\varepsilon)\sqrt{\lambda_{m}})}{\cosh(\varepsilon\sqrt{\lambda_{m}})}\psi_{\lambda_{m}}(y_{1}),\\
\PP_{2}&=\sum_{m=0}^{\infty}\frac{\bar{\varphi}_{3,m}(0)\cosh(y_{1}\sqrt{\mu_{m}})-\bar{\varphi}_{1,m}(0)\cosh((l-y_{1})\sqrt{\mu_{m}})}{\sqrt{\mu_{m}}\sinh(l\sqrt{\mu_{m}})}
\psi_{\mu_{m}}(y_{2}),
\end{align*}
where $\lambda_{m}=\lambda_{m}(\varepsilon)$ and $\mu_{m}=\mu_{m}(l)$ are defined with \eqref{2.0}, and
 $\bar{\varphi}_{1,m}(t)=\langle \bar{\varphi}_{1},\psi_{\mu_{m}}\rangle_{\varepsilon},$
$\bar{\varphi}_{2,m}(t)=\langle \bar{\varphi}_{2},\psi_{\lambda_{m}}\rangle_{l},$
$\bar{\varphi}_{3,m}(t)=\langle \bar{\varphi}_{3},\psi_{\mu_{m}}\rangle_{\varepsilon},$ we assert the following result.

\begin{lemma}\label{l5.1}
Let $\alpha\in(0,1)$, $\varepsilon>0$ be arbitrarily fixed and let
assumptions \textbf{(h1)-(h3)} hold. Then boundary-value problem
\eqref{5.1} admits a unique classical solution
\begin{equation}\label{5.2}
p_{0}^{\varepsilon}=\PP_{0}+\PP_{1}+\PP_{2}
\end{equation}
in the domain $\Omega^{\varepsilon}$, satisfying the regularity
$
p^{\varepsilon}_{0}\in\C^{2+\alpha}(\bar{\Omega}^{\varepsilon}),
$
and
\[
\|p^{\varepsilon}_{0}\|_{\C^{2+\alpha}(\bar{\Omega}^{\varepsilon})}\leq
C[ \|\varphi_{1}\|_{\CC^{2+\alpha}(\Gamma^{\varepsilon}_{1,T})}+
\|\varphi_{2}\|_{\CC^{2+\alpha}(\Gamma_{2,T})}+\|\varphi_{3}\|_{\CC^{2+\alpha}(\Gamma^{\varepsilon}_{3,T})}].\]
Besides, under condition
\begin{equation}\label{3.3}
\sqrt{\frac{2}{l}}\sum_{m=0}^{\infty}\frac{(-1)^{m+1}[\bar{\varphi}_{3,m}(0)\cosh(y_{1}\sqrt{\mu_{m}})-
\bar{\varphi}_{1,m}(0)\cosh((l-y_{1})\sqrt{\mu_{m}})]}{\sinh(l\sqrt{\mu_{m}})}
+\sum_{m=0}^{\infty}\frac{\bar{\varphi}_{2,m}(0)\psi_{\lambda_{m}}(y_{1})}{\cosh(\varepsilon\sqrt{\lambda_{m}})}
<0,
\end{equation}
 the function $p^{\varepsilon}_{0}$
satisfies the inequality
\begin{equation}\label{5.3}
\frac{\partial p_{0}^{\varepsilon}}{\partial y_{2}}<
0\quad\text{on}\quad\Gamma^{\varepsilon},
\end{equation}
which provides estimate \eqref{3.4}.

 If in addition
$\alpha\in(1/2,1)$, then
$p^{\varepsilon}_{0}\in\C^{3}(\bar{\Omega}^{\varepsilon})$ and
\begin{equation}\label{5.4}
\|p^{\varepsilon}_{0}\|_{\C^{3}(\bar{\Omega}^{\varepsilon})}\leq
C[ \|\varphi_{1}\|_{\CC^{2+\alpha}(\Gamma^{\varepsilon}_{1,T})}+
\|\varphi_{2}\|_{\CC^{2+\alpha}(\Gamma_{2,T})}+\|\varphi_{3}\|_{\CC^{2+\alpha}(\Gamma^{\varepsilon}_{3,T})}].
\end{equation}
\end{lemma}
\begin{proof}
Taking into account the smoothness of $\bar{\varphi}_{i}$ and using standard Fourier approach, we construct, at least formally, a solution of \eqref{5.1}
 in  form \eqref{5.2}. The direct calculations provide that $\PP_{0}\in\C^{3+\alpha}(\bar{\Omega}^{\varepsilon})$.
Next, it is apparent that the functions  $\bar{\varphi}_{i}$ meet requirements of Proposition \ref{p2.1}, and hence
 for each fixed $\varepsilon>0,$ the serieses $\PP_{1}$ and $\PP_{2}$ are convergent absolutely and uniformly
  in $\C^{2+\alpha}(\bar{\Omega}^{\varepsilon})$ if $\alpha\in(0,1)$ and in $\C^{3}(\bar{\Omega}^{\varepsilon})$ if $\alpha\in(1/2,1)$.
  Thus, we arrive at the estimates
\begin{align*}
\sum_{j=0}^{2}\|\PP_{j}\|_{\C^{2+\alpha}(\bar{\Omega}^{\varepsilon})}&\leq
C[\|\bar{\varphi}_{1}(\cdot,0)\|_{\C^{2+\alpha}(\Gamma^{\varepsilon}_{1})}
+\|\bar{\varphi}_{2}(\cdot,0)\|_{\C^{2+\alpha}(\Gamma_{2})}
+\|\bar{\varphi}_{3}(\cdot,0)\|_{\C^{2+\alpha}(\Gamma^{\varepsilon}_{3})}]
\\
&\leq
C[\|\chi_{2}\varphi_{1}\|_{\CC^{2+\alpha}(\Gamma^{\varepsilon}_{1,T})} +
\|\chi_{1}\varphi_{2}\|_{\CC^{2+\alpha}(\Gamma_{2,T})}
+\|\chi_{2}\varphi_{3}\|_{\CC^{2+\alpha}(\Gamma^{\varepsilon}_{3,T})}]
\quad\text{if}\quad \alpha\in(0,1),\\
\sum_{j=0}^{2}\|\PP_{j}\|_{\C^{3}(\bar{\Omega}_{\varepsilon})}&\leq
 C[\|\chi_{2}\varphi_{1}\|_{\CC^{2+\alpha}(\Gamma^{\varepsilon}_{1,T})} +
\|\chi_{1}\varphi_{2}\|_{\CC^{2+\alpha}(\Gamma_{2,T})}
+\|\chi_{2}\varphi_{3}\|_{\CC^{2+\alpha}(\Gamma^{\varepsilon}_{3,T})}]\quad\text{if}\quad
\alpha\in(1/2,1)
\end{align*}
with the constant $C$ is independent of $\varepsilon$ for $\varepsilon\in(0,1)$.

Finally, the straightforward calculations together with the obtained regularity of $\PP_{j}$, provides that the function $\PP_{0}+\PP_{1}+\PP_{2}$
 satisfies the equation and the boundary conditions in \eqref{5.1}. Thus, representation \eqref{5.2} and estimates of $\PP_{j}$ provide coercive estimates for
 $p^{\varepsilon}_{0}$, in particular  \eqref{5.4}. Uniqueness of the constructed solution follows immediately from the coercive estimate for $p^{\varepsilon}_{0}$.

At last, to finish the proof of Lemma \ref{l5.1}, we are left to verify \eqref{5.3}. The direct calculations and representation \eqref{5.2} arrive at
\begin{align*}
\frac{\partial p^{\varepsilon}_{0}}{\partial y_{2}}&=\sum_{j=0}^{2}\frac{\partial\PP_{j}}{\partial y_{2}}=\varepsilon^{-1/2}\bar{\varphi}_{2,0}(0) +
\sum_{m=1}^{\infty}\bar{\varphi}_{2,m}(0)\frac{\cosh((y_{2}-\varepsilon)\sqrt{\lambda_{m}})}{\cosh(\varepsilon\sqrt{\lambda_{m}})}\psi_{\lambda_{m}}(y_{1})\\
&
-
\sum_{m=0}^{\infty}\frac{[\bar{\varphi}_{3,m}(0)\cosh(y_{1}\sqrt{\mu_{m}})-\bar{\varphi}_{1,m}(0)\cosh((y_{1}-l)\sqrt{\mu_{m}})]}{\sinh(l\sqrt{\mu_{m}})}\psi_{\mu_{m}}(y_{2}).
\end{align*}
Then, substituting $y_{2}=\varepsilon$ to this representation and taking into account \eqref{3.3}, we end up with estimate \eqref{5.3}.

Besides, the second boundary condition in \eqref{1.2} on
$\Gamma^{\varepsilon}(t)$ together with \eqref{5.3} provide for
$t=0$
\[
V_{n}\Big|_{\Gamma^{\varepsilon}}=-\gamma^{-1}\frac{\partial
p^{\varepsilon}_{0}}{\partial n}>0
\]
if \eqref{3.3} holds. This completes the proof of this statement.
\end{proof}

At this point, we show that the constructed solution in Lemma \ref{l5.1} is  more regular. To this end, it is enough to apply Theorem 3.1 \cite{Vo} to \eqref{5.1}.
\begin{lemma}\label{l5.2}
Let $\alpha\in(0,1)$ and assumptions \textbf{(h1)-(h3)} hold. Then the
classical solution $p^{\varepsilon}_{0}$ belongs to
$\C^{3+\alpha}(\bar{\Omega}^{\varepsilon})$ and
$$
\|p^{\varepsilon}_{0}\|_{\C^{3+\alpha}(\bar{\Omega}^{\varepsilon})}\leq
C[ \|\chi_{2}\varphi_{1}\|_{\CC^{2+\alpha}(\Gamma^{\varepsilon}_{1,T})}+
\|\chi_{1}\varphi_{2}\|_{\CC^{2+\alpha}(\Gamma_{2,T})}+\|\chi_{2}\varphi_{3}\|_{\CC^{2+\alpha}(\Gamma^{\varepsilon}_{3,T})}],
$$
where the constant $C$ is independent of $\varepsilon$ if $\varepsilon\in(0,1)$.

Besides, the function
$
s(y_{1},t)=-\frac{t}{\gamma}\frac{\partial
p^{\varepsilon}_{0}}{\partial y_{2}}\Big|_{\Gamma^{\varepsilon}}
$
satisfies relations
\[
s(y_{1},0)=0,\quad \frac{\partial s}{\partial t}(y_{1},0)=V_{\mathbf{n}}\Big|_{t=0}\quad
\text{on}\quad \Gamma^{\varepsilon},
\]
\[
\|s\|_{\CC^{2+\alpha}(\Gamma^{\varepsilon}_{
T})}+\|\partial s/\partial t\|_{\CC^{2+\alpha}(\Gamma^{\varepsilon}_{ T})}\leq C
[\|\chi_{2}\varphi_{1}\|_{\CC^{2+\alpha}(\Gamma^{\varepsilon}_{1,T})}+
\|\chi_{1}\varphi_{2}\|_{\CC^{2+\alpha}(\Gamma_{2,T})}+
\|\varphi_{3}\chi_{2}\|_{\CC^{2+\alpha}(\Gamma^{\varepsilon}_{3,T})}].
\]
\end{lemma}

Note that the statements related to the function $s$ follow
immediately from the properties of the function
$p^{\varepsilon}_{0}$.


\subsection{Reducing problem \eqref{1.1}-\eqref{1.2} to a nonlinear problem in $\Omega^{\varepsilon}\times(0,T)$}
\label{s5.2}

\noindent
Denoting
the cut-off function by $\chi(\Lambda)\in \C_{0}^{\infty}(\R)$ such that $0\leq\chi(\Lambda)\leq 1$ and
\[
\chi(\Lambda)=
\begin{cases}
1,\quad\text{if}\quad |\Lambda|<\varepsilon/15,\\
0,\quad\text{if}\quad |\Lambda|>2\varepsilon/15,
\end{cases}
\]
we introduce the new coordinate
\[
y_{1}=x_{1}\quad \text{and}\quad
y_{2}=x_{2}+\rho(x_{1},t)\chi(\Lambda),
\]
with $\Lambda=x_{2}-\varepsilon$.

It is apparent that (see for details, e.g., Section 3 in \cite{V2} or Section 6 in \cite{BF1}), this transformation reduces the domain
 $\Omega^{\varepsilon}(t)$, $t\in(0,T),$ to the fixed domain $\Omega^{\varepsilon}_{ T}=\Omega^{\varepsilon}\times (0,T)$.

After that, introducing a new unknown function
\[
v=v(x_{1},x_{2},t)=p^{\varepsilon}(x_{1},y_{2}(x_{1},x_{2},t),t),
\]
we rewrite the equation in \eqref{1.2} in the new  function and variables
\begin{equation}\label{5.5}
\Delta v+2\frac{\partial x_{2}}{\partial
y_{1}}\frac{\partial^{2}v}{\partial x_{1}\partial x_{2}}
+\Big[\Big(\frac{\partial x_{2}}{\partial
y_{1}}\Big)^{2}+\Big(\frac{\partial x_{2}}{\partial
y_{2}}\Big)^{2}-1\Big]\frac{\partial^{2}v}{\partial
x_{2}^{2}}+\Big(\frac{\partial^{2} x_{2}}{\partial
y_{1}^{2}}+\frac{\partial^{2} x_{2}}{\partial y_{2}^{2}}\Big)
\frac{\partial v}{\partial x_{2}}=0\quad\text{in}\quad
\Omega^{\varepsilon}_{T},
\end{equation}
where we set
\begin{equation}\label{5.5*}
\begin{cases}
\dfrac{\partial x_{1}}{\partial y_{1}}=1,\quad \dfrac{\partial x_{1}}{\partial y_{2}}=0,\\\\
\dfrac{\partial x_{2}}{\partial y_{1}}=-\dfrac{\chi\frac{\partial \rho}{\partial x_{1}}}{1+\chi'\rho},
\quad \dfrac{\partial x_{2}}{\partial y_{2}}=\dfrac{1}{1+\chi'\rho},\\\\
\dfrac{\partial^{2} x_{2}}{\partial y_{2}^{2}}=-\dfrac{\chi''\rho}{(1+\chi'\rho)^{3}},\\\\
\dfrac{\partial^{2} x_{2}}{\partial y_{1}^{2}}=\dfrac{\chi\chi''\rho\frac{\partial \rho}{\partial x_{1}}}{(1+\chi'\rho)^{3}}+
\dfrac{\chi[\chi'(\frac{\partial \rho}{\partial x_{1}})^{2}-\frac{\partial^{2} \rho}{\partial x^{2}_{1}}(2+\chi'\rho)]}{(1+\chi'\rho)^{2}}.
\end{cases}
\end{equation}

At this point, we begin to rewrite the conditions on the free boundary in the new variables.
Recasting the arguments from Section \ref{s4} leading to representation \eqref{4.2}, we deduce that the Stefan condition of the moving boundary has the form
\[
\gamma\frac{\partial \rho}{\partial t}=\frac{\partial p^{\varepsilon}}{\partial
y_{1}}\frac{\partial \rho}{\partial y_{1}}-\frac{\partial
p^{\varepsilon}}{\partial y_{2}}.
\]

As a result, taking into account \eqref{5.5*}, we can can rewrite
the boundary conditions on $\Gamma^{\varepsilon}(t)$ in the form
\begin{equation}\label{5.6}
\begin{cases}
v(x_{1},x_{2},t)=0\quad\text{on}\quad \Gamma^{\varepsilon}_{T},\\
\gamma\dfrac{\partial \rho}{\partial t}=\dfrac{\partial\rho}{\partial x_{1}}\dfrac{\partial v}{\partial x_{1}}-\Big[1+\Big(\dfrac{\partial\rho}{\partial x_{1}}\Big)^{2}\Big]\dfrac{\partial v}{\partial x_{2}} \quad\text{on}\quad \Gamma^{\varepsilon}_{ T},\\
\rho(x_{1},0)=0\quad\text{in}\quad [0,l].
\end{cases}
\end{equation}
Finally, in virtue of assumptions \textbf{(h2)-(h3)} and the definition of $\chi(\Lambda)$, the rest boundary conditions in \eqref{1.2} remain unchanged:
\begin{equation}\label{5.7}
\begin{cases}
\dfrac{\partial v}{\partial x_{1}}=-\chi_{2}(x_{2}/\varepsilon)\varphi_{1}(x_{2}/\varepsilon,t)\quad\text{on}\quad\Gamma_{1,T}^{\varepsilon},\\
\dfrac{\partial v}{\partial x_{2}}=-\varepsilon\chi_{1}(x_{1})\varphi_{2}(x_{1},t)\quad\text{on}\quad\Gamma_{2,T},\\
\dfrac{\partial v}{\partial
x_{1}}=\chi_{2}(x_{2}/\varepsilon)\varphi_{3}(x_{2}/\varepsilon,t)\quad\text{on}\quad\Gamma_{3,T}^{\varepsilon}.
\end{cases}
\end{equation}

Summing up, we can reformulate Theorem \ref{t3.1} as follows.
\begin{theorem}\label{t5.1}
\textbf{(Reformulated Theorem \ref{t3.1})} Let conditions of
Theorem \ref{t3.1} hold. Then for some small $T$ and each fixed
positive $\varepsilon$, there exists a unique solution
$(v(x_{1},x_{2},t),\rho(x_{1},t))$ of nonlinear problem
\eqref{5.5}-\eqref{5.7} satisfying regularity
$
v\in\CC^{2+\alpha}(\bar{\Omega}^{\varepsilon}_{ T}),\quad
\rho(x_{1},t)\in\hat{\CC}^{2+\alpha}(\Gamma^{\varepsilon}_{ T}).
$
Besides,
\begin{equation}\label{5.8}
v(x_{1},x_{2},0)=p^{\varepsilon}_{0}(x_{1},x_{2})\quad\text{in}\quad\bar{\Omega}^{\varepsilon},
\end{equation}
where $p^{\varepsilon}_{0}$ is given with \eqref{5.1}.
\end{theorem}

Thus the proof of Theorem \ref{t3.1} is equivalent to the one of Theorem \ref{t5.1}.
The rest part of this section is devoted to the proof of Theorem \ref{t5.1}. It is worth mentioning that equality \eqref{5.8} follows immediately from \eqref{5.5}-\eqref{5.6}.


\subsection{A perturbation form of system \eqref{5.5}-\eqref{5.7}}
\label{s5.3}

\noindent
In this subsection, we linearize system \eqref{5.5}-\eqref{5.7} on the initial data and rewrite the one in the form
\[
\A\mathbf{z}=\mathcal{F}\mathbf{z},
\]
where $\A$ is a linear operator, while $\mathcal{F}$ is a nonlinear perturbation. To this end, we introduce new unknown functions
\begin{equation}\label{5.9}
\begin{cases}
\sigma=\sigma(x_{1},t)=\rho(x_{1},t)-s(x_{1},t),\\
u=u(x_{1},x_{2},t)=v(x_{1},x_{2},t)-p^{\varepsilon}_{0}(x_{1},x_{2})-\chi(\Lambda)\dfrac{\partial
p^{\varepsilon}_{0}}{\partial x_{2}}(x_{1},x_{2})\sigma(x_{1},t),
\end{cases}
\end{equation}
where $p^{\varepsilon}_{0}$ solves problem \eqref{5.1} and $s$ is
defined in Lemma \ref{l5.2}.

Then, substituting \eqref{5.9} to \eqref{5.5}-\eqref{5.7}, we end up with
\begin{equation}\label{5.10}
\begin{cases}
\Delta u=\mathcal{F}_{0}(u,\sigma)\quad\text{in}\quad \Omega^{\varepsilon}_{ T},\\
u=-\dfrac{\partial p^{\varepsilon}_{0}}{\partial x_{2}}\sigma\quad\text{on}\quad \Gamma^{\varepsilon}_{ T},\\
\gamma\dfrac{\partial\sigma}{\partial t}+\dfrac{\partial u}{\partial x_{2}}=\mathcal{F}_{1}(u,\sigma)\quad\text{on}\quad \Gamma^{\varepsilon}_{T},\\
\dfrac{\partial u}{\partial x_{1}}=0\quad\text{on}\quad \Gamma^{\varepsilon}_{1,T}\cup\Gamma^{\varepsilon}_{3,T},\\
\dfrac{\partial u}{\partial x_{2}}=0\quad\text{on}\quad \Gamma_{2,T},\\
\sigma(x_{1},0)=0\quad\text{in}\quad [0,l],\\
u(x_{1},x_{2},0)=0\quad\text{in}\quad \bar{\Omega}^{\varepsilon},
\end{cases}
\end{equation}
where we put
\begin{align*}
-\mathcal{F}_{0}(u,\sigma)&=2\frac{\partial
x_{2}}{\partial y_{1}}
\frac{\partial^{2}}{\partial x_{1}\partial x_{2}}\Big(u+p^{\varepsilon}_{0}+\chi \sigma\frac{\partial p^{\varepsilon}_{0}}{\partial x_{2}}\Big) +\Big[\Big(\frac{\partial x_{2}}{\partial
y_{1}}\Big)^{2}+\Big(\frac{\partial x_{2}}{\partial
y_{2}}\Big)^{2}-1\Big]
\frac{\partial^{2}}{\partial x_{2}^{2}}\Big(u+p^{\varepsilon}_{0}+\chi\sigma\frac{\partial p^{\varepsilon}_{0}}{\partial x_{2}}\Big)\\\\
& +\Big(\frac{\partial^{2}x_{2} }{\partial
y_{1}^{2}}+\frac{\partial^{2}x_{2} }{\partial y_{2}^{2}}\Big)
\frac{\partial}{\partial x_{2}}
\Big(u+p^{\varepsilon}_{0}+\chi\sigma\frac{\partial
p^{\varepsilon}_{0}}{\partial x_{2}}\Big)+\Delta\Big(\chi\sigma\frac{\partial
p^{\varepsilon}_{0}}{\partial x_{2}}\Big)
\end{align*}
with $\frac{\partial x_{i}}{\partial y_{j}}$, $\frac{\partial^{2} x_{i}}{\partial y_{j}^{2}}$ are given by \eqref{5.5*} and depended on $\sigma$ through relation \eqref{5.9},
\begin{equation*}
\mathcal{F}_{1}(u,\sigma)=-\Big[1+\Big(\frac{\partial\sigma}{\partial x_{1}}+\frac{\partial s}{\partial x_{1}}\Big)^{2}\Big]\frac{\partial^{2} p^{\varepsilon}_{0}}{\partial x_{2}^{2}}\sigma
 -\Big(\frac{\partial \sigma}{\partial x_{1}}+\frac{\partial s}{\partial x_{1}}\Big)^{2}\Big(\frac{\partial u}{\partial
x_{2}}+\frac{\partial p^{\varepsilon}_{0}}{\partial x_{2}}\Big).
\end{equation*}
Thus, system \eqref{5.5}-\eqref{5.7} is written in the short convenient form
\[
\A\mathbf{z}=\mathcal{F}\mathbf{z},\quad \mathbf{z}=(u,\sigma).
\]
Based on Lemma \ref{l5.2}, boundary conditions in \eqref{5.10} and representations of $\mathcal{F}_{0}$ and $\mathcal{F}_{1}$, we assert the following result.
\begin{corollary}\label{c5.1}
The functions $\mathcal{F}_{0}(u,\sigma)$ and $\mathcal{F}_{1}(u,\sigma)$ contain the higher derivatives of $u$ and $\sigma$ with coefficients that tends to zero as $t\to 0$; the ''quadratic'' terms with respect to $u$ and $\sigma$, and their derivatives; and the terms of minor differential orders of unknown functions. Besides,
\begin{align*}
\mathcal{F}_{0}(u,\sigma)\Big|_{t=0}&=\mathcal{F}_{1}(u,\sigma)\Big|_{t=0}=0,\\
\mathcal{F}_{0}(u,\sigma)&=0\quad\text{at}\quad(x_{1},x_{2})\in\{(0,0);(0,\varepsilon);(l,0);(l,\varepsilon)\},\quad t\in[0,T],\\
\mathcal{F}_{1}(u,\sigma)&=0\quad\text{at}\quad(x_{1},x_{2})\in\{(0,\varepsilon);(l,\varepsilon)\},\quad t\in[0,T],
\end{align*}
\begin{align*}
\mathcal{F}_{0}(0,0)&=\frac{2\chi \frac{\partial s}{\partial x_{1}}}{1+\chi'
s}\frac{\partial^{2}p^{\varepsilon}_{0}}{\partial x_{1}\partial
x_{2}}-
\Big[\frac{\chi^{2} ( \frac{\partial s}{\partial x_{1}})^{2}+1}{(1+\chi' s)^{2}}-1\Big]\frac{\partial^{2}p^{\varepsilon}_{0}}{\partial x_{2}^{2}}\\
& -\Big[\frac{\chi\chi''s  \frac{\partial s}{\partial x_{1}}-s\chi''}{(1+\chi' s)^{3}}+
\frac{\chi[\chi'( \frac{\partial s}{\partial x_{1}})^{2}- \frac{\partial^{2} s}{\partial x^{2}_{1}}(2+\chi's)]}{(1+\chi'
s)^{2}}
\Big]\frac{\partial p^{\varepsilon}_{0}}{\partial x_{2}},\\
\mathcal{F}_{1}(0,0)&=-\Big( \frac{\partial s}{\partial x_{1}}\Big)^{2}\frac{\partial
p^{\varepsilon}_{0}}{\partial x_{2}}.
\end{align*}
\end{corollary}

In the next step, we show the boundedness of the linear operator $\A$ in the corresponding functional spaces. To this end, freezing  the functional arguments in the functions $\mathcal{F}_{0}(u,\sigma)$ and $\mathcal{F}_{1}(u,\sigma)$, we obtain from \eqref{5.10} the linear system with variable coefficients, which will be analyzed in detail in Subsection \ref{s5.5}. It is worth mentioning that the model problem with a dynamic boundary condition plays  a key point in the
investigation of this linear system.


\subsection{Model problem in the right angle}
\label{s5.4}

\noindent
In order to construct the model problem near the boundary $\Gamma_{\varepsilon T}$ by using the Schauder approach, it is necessary to fix the coefficients of the original problem at the boundary point. In this section, we study the boundary-value problem with a dynamic boundary condition in the right angle. Namely, let $\mathfrak{C}_{0}$ be some positive number and
\begin{align*}
\mathfrak{R}&=\{(x_{1},x_{2})\in\R^{2}:\quad x_{1}>0,\quad x_{2}>0\},\quad
\mathfrak{R}_{T}=\mathfrak{R}\times(0,T),\\
\mathfrak{R}_{1}&=\{(x_{1},x_{2})\in\R^{2}:\quad x_{1}=0,\quad x_{2}\geq 0\},\quad
\mathfrak{R}_{1,T}=\mathfrak{R}_{1}\times[0,T],\\
\mathfrak{R}_{2}&=\{(x_{1},x_{2})\in\R^{2}:\quad x_{2}=0,\quad x_{1}\geq 0\},\quad
\mathfrak{R}_{2,T}=\mathfrak{R}_{2}\times[0,T],\\
\end{align*}
We consider the initial-boundary problem in the unknown function $U=U(x,t):\mathfrak{R}_{T}\to\R$
\begin{equation}\label{5.11}
\begin{cases}
\Delta U=0\quad\text{in}\quad \mathfrak{R}_{T},\\
\dfrac{\partial U}{\partial x_{1}}=0\quad\text{on}\quad \mathfrak{R}_{1,T},\\
\dfrac{\partial U}{\partial t}-\mathfrak{C}_{0}\dfrac{\partial U}{\partial x_{2}}=f(x_{1},t)\quad\text{on}\quad \mathfrak{R}_{2,T},\\
U=0\quad\text{if}\quad |x|\to\infty,\quad t\in[0,T],\\
U(x,0)=0\quad\text{in}\quad \bar{\mathfrak{R}},
\end{cases}
\end{equation}
where $f$ is a given function satisfying conditions

\noindent
\textbf{(h6):} for some positive number $r$, $f\equiv 0$ if either $t\leq 0$ or $|x|>r$,  and
$
f\in\CC^{1+\alpha}(\bar{\mathfrak{R}}_{2,T}).
$

\begin{lemma}\label{l5.3}
Under assumption \textbf{(h6)} problem \eqref{5.1} admits a unique classical solution $U$ in $\bar{\mathfrak{R}}_{T}$ satisfying regularity
$U\in\CC^{2+\alpha}(\bar{\mathfrak{R}}_{T})\quad\text{and}\quad
\frac{\partial U}{\partial
t}\in\CC^{1+\alpha}(\mathfrak{R}_{2,T}).
$
Besides, the estimate holds
$$\|U\|_{\CC^{2+\alpha}(\bar{\mathfrak{R}}_{T})}+\|\partial U/\partial t\|_{\CC^{1+\alpha}(\mathfrak{R}_{2,T})}\leq C\|f\|_{\CC^{1+\alpha}(\mathfrak{R}_{2,T})}.$$
\end{lemma}
\begin{proof}
First of all, taking into account the homogenous Neumann boundary
condition on $\mathfrak{R}_{1,T}$ (which can be considered as a
symmetry condition), we can study  instead of  problem
\eqref{5.11} in the right angle $\mathfrak{R}_{T}$ the similar
problem in the upper semi-space
\[
\R^{2}_{+T}=\R^{2}_{+}\times(0,T), \quad\R^{2}_{+}=\{(x_{1},x_{2})\in\R^{2}:\quad x_{1}\in\R,\quad x_{2}>0\}.
\]
Namely, introducing a new function
\[
F(x_{1},t)=
\begin{cases}
f(x_{1},t),\quad x_{1}\geq 0,\\
f(-x_{1},t),\quad x_{1}<0,
\end{cases}
\]
we consider the new problem in the unknown function $\mathcal{U}=\mathcal{U}(x,t):\R^{2}_{+T}\to\R$:
\begin{equation}\label{5.12}
\begin{cases}
\Delta \mathcal{U}=0\quad\text{in}\quad \R^{2}_{+T},\\
\dfrac{\partial \mathcal{U}}{\partial t}-\mathfrak{C}_{0}\dfrac{\partial \mathcal{U}}{\partial x_{2}}=F(x_{1},t),\quad\text{on}\quad \R_{T},\\
\mathcal{U}=0\quad\text{if}\quad |x|\to\infty,\quad t\in[0,T],\\
\mathcal{U}(x,0)=0\quad\text{in}\quad \bar{\R}^{2}_{+T}.
\end{cases}
\end{equation}
It is apparent that,

\noindent
$\bullet$ \textit{the function $F$ meets the requirement \textbf{(h6)} and}
\[
\|F\|_{\CC^{1+\alpha}(\bar{\R}_{T})}\leq
C\|f\|_{\CC^{1+\alpha}(\bar{\mathfrak{R}}_{2,T})};
\]

\noindent
$\bullet$ \textit{the solution $\mathcal{U} $ of problem \eqref{5.12} in $\bar{\mathfrak{R}}_{T} $
boils down with the solution $U$ of \eqref{5.11}, i.e.}
$$\mathcal{U}(x,t)\Big|_{\bar{\mathfrak{R}}_{T}}=U(x,t).$$
Thus, it is enough to prove statements of Lemma \ref{l5.3} to  problem \eqref{5.12}.

To this end, applying standard Fourier and Laplace transformation with respect to $x_{1}$ and $t$, correspondingly, we construct the  integral representation of the solution to \eqref{5.12}
\[
\mathcal{U}(x,t)=\int_{0}^{t}d\tau\int_{-\infty}^{+\infty}  F(t-\tau,x_{1}-\zeta)K(\zeta,\tau)d\zeta
\]
with the kernel $K$ defined with \eqref{2.3}.

After that, taking advantage of Lemma \ref{l2.1} and recasting the arguments of Chapter 4 in \cite{LSU}, we arrive at the estimate
\[
\|\mathcal{U}\|_{\CC^{2+\alpha}(\bar{\R}^{2}_{+T})}+\|\partial \mathcal{U}/\partial t\|_{\CC^{1+\alpha}(\bar{\R}_{T})}
\leq C\|F\|_{\CC^{1+\alpha}(\bar{\R}_{T})}
\leq C\|f\|_{\CC^{1+\alpha}(\mathfrak{R}_{2,T})}.
\]

Then, substituting the integral representation of $\mathcal{U}$ to the equation, initial and boundary conditions in \eqref{5.12}, and using Lemma \ref{l2.1}, we conclude that the constructed function $\mathcal{U}$ satisfies  all the relations in \eqref{5.12} in the classical sense.

Finally, we note that, the coercive estimate of $\mathcal{U}$ provides the uniqueness of the solution to \eqref{5.12}.
That completes the proof of Lemma \ref{l5.3}.
\end{proof}


\subsection{The one-valued solvability of the linear system $\A \mathbf{z}=\mathcal{F}$}
\label{s5.5}

\noindent
As it follows from \eqref{5.10}, the linear system corresponding to the nonlinear one
has the form
\begin{equation}\label{5.13}
\begin{cases}
\Delta u=f_{0}(x,t)\quad\text{in}\quad \Omega^{\varepsilon}_{ T},\\
u=A(x_{1})\sigma\quad\text{on}\quad \Gamma^{\varepsilon}_{T},\\
\gamma\dfrac{\partial\sigma}{\partial t}+\dfrac{\partial u}{\partial x_{2}}=f_{1}(x,t)\quad\text{on}\quad \Gamma^{\varepsilon}_{ T},\\
\dfrac{\partial u}{\partial x_{1}}=0\quad\text{on}\quad \Gamma^{\varepsilon}_{1,T}\cup\Gamma^{\varepsilon}_{3,T},\\
\dfrac{\partial u}{\partial x_{2}}=0\quad\text{on}\quad \Gamma_{2,T},\\
\sigma(x_{1},0)=0\quad\text{in}\quad [0,l],\\
u(x_{1},x_{2},0)=0\quad\text{in}\quad \bar{\Omega}^{\varepsilon}.
\end{cases}
\end{equation}
Here $A(x_{1}),$ $f_{0}(x,t)$, $f_{1}(x,t)$ are some given functions satisfying the following conditions:

\noindent
\textbf{(h7)(Consistency conditions):}
\begin{align*}
f_{0}(x_{1},x_{2},0)&=f_{0}(0,0,t)=f_{0}(l,0,t)=f_{0}(l,\varepsilon,t)=f_{0}(0,\varepsilon, t)=0\quad\text{for}\quad (x_{1},x_{2})\in\bar{\Omega}^{\varepsilon}, \, t\in[0,T];\\
f_{1}(x_{1},x_{2},0)&=f_{1}(0,\varepsilon,t)=f_{1}(l,\varepsilon,t)=0\quad\text{for}\quad
(x_{1},x_{2})\in\Gamma^{\varepsilon}, \, t\in[0,T];
\end{align*}
the function $A(x_{1})$ is strictly positive, i.e. $A\geq\delta>0$ for all $x_{1}\in[0,l]$.

\noindent
\textbf{(h8)(Regularity of the given functions):}
\[
f_{0}\in\CC^{\alpha}(\bar{\Omega}^{\varepsilon}_{ T}),\qquad
f_{1}\in\CC^{1+\alpha}(\Gamma^{\varepsilon}_{ T}),\qquad
A\in\C^{2+\alpha}([0,l]).
\]
\begin{theorem}\label{t5.2}
Under conditions \textbf{(h1), (h7)} and \textbf{(h8)}, for some $T>0$ and any fixed
positive $\varepsilon$, problem \eqref{5.13} admits a unique
classical solution $(u,\sigma)$ satisfying the regularity
$
u\in\CC^{2+\alpha}(\bar{\Omega}^{\varepsilon}_{
T})$ and
$\sigma\in\hat{\CC}^{2+\alpha}(\Gamma^{\varepsilon}_{T}),
$
and the estimate
\[
\|u\|_{\CC^{2+\alpha}(\bar{\Omega}^{\varepsilon}_{
T})}+\|\sigma\|_{\hat{\CC}^{2+\alpha}(\Gamma^{\varepsilon}_{
T})}\leq C[\|f_{0}\|_{\CC^{\alpha}(\bar{\Omega}^{\varepsilon}_{
T})}+ \|f_{1}\|_{\CC^{1+\alpha}(\Gamma^{\varepsilon}_{T})} ]
\]
with the constant $C$ independent of the right-hand sides in \eqref{5.13}.
\end{theorem}
\begin{proof}
It is convenient to reduce linear system \eqref{5.13} to the same
problem with homogenous equation. To this end, we apply Theorem
3.1 \cite{Vo} to the following linear problem with the unknown
function $\mathfrak{U}=\mathfrak{U}(x,t):\Omega^{\varepsilon}_{
T}\to\R$
\begin{equation*}
\begin{cases}
\Delta \mathfrak{U}=f_{0}(x,t)\quad\text{in}\quad \Omega^{\varepsilon}_{T},\\
\mathfrak{U}=0\quad\text{on}\quad \Gamma^{\varepsilon}_{ T},\\
\dfrac{\partial \mathfrak{U}}{\partial x_{1}}=0\quad\text{on}\quad \Gamma^{\varepsilon}_{1,T}\cup\Gamma^{\varepsilon}_{3,T},\\
\dfrac{\partial \mathfrak{U}}{\partial x_{2}}=0\quad\text{on}\quad \Gamma_{2,T},\\
\mathfrak{U}(x_{1},x_{2},0)=0\quad\text{in}\quad
\bar{\Omega}^{\varepsilon},
\end{cases}
\end{equation*}
and deduce the existence of a unique solution
$\mathfrak{U}\in\CC^{2+\alpha}(\bar{\Omega}^{\varepsilon}_{ T})$
satisfying relations
\[
\frac{\partial \mathfrak{U}}{\partial t}=0\quad\text{on}\quad \Gamma^{\varepsilon}_{ T},\quad
\|\mathfrak{U}\|_{\CC^{2+\alpha}(\bar{\Omega}^{\varepsilon}_{
T})}\leq C \|f_{0}\|_{\CC^{\alpha}(\bar{\Omega}^{\varepsilon}_{
T})}.
\]
Then we look for a solution to the original problem \eqref{5.13} in the form
\[
u=\mathfrak{U}+w,
\]
where the unknown function $w$ solves the problem
\begin{equation}\label{5.14}
\begin{cases}
\Delta w=0\quad\text{in}\quad \Omega^{\varepsilon}_{ T},\\
w=A(x_{1})\sigma\quad\text{on}\quad \Gamma^{\varepsilon}_{ T},\\
\gamma\dfrac{\partial\sigma}{\partial t}+\dfrac{\partial w}{\partial x_{2}}=\bar{f}_{1}(x,t)\quad\text{on}\quad \Gamma^{\varepsilon}_{ T},\\
\dfrac{\partial w}{\partial x_{1}}=0\quad\text{on}\quad \Gamma^{\varepsilon}_{1,T}\cup\Gamma^{\varepsilon}_{3,T},\\
\dfrac{\partial w}{\partial x_{2}}=0\quad\text{on}\quad \Gamma_{2,T},\\
\sigma(x_{1},0)=0\quad\text{in}\quad [0,l],\\
w(x_{1},x_{2},0)=0\quad\text{in}\quad \bar{\Omega}^{\varepsilon}.
\end{cases}
\end{equation}
Here $
\bar{f}_{1}(x,t)=f(x,t)-\frac{\partial \mathfrak{U}}{\partial x_{2}}\Big|_{\Gamma^{\varepsilon}_{ T}}$ and
\[
\|\bar{f}_{1}\|_{\CC^{1+\alpha}(\Gamma^{\varepsilon}_{ T})}\leq
C[\|f_{1}\|_{\CC^{1+\alpha}(\Gamma^{\varepsilon}_{ T})}+
\|f_{0}\|_{\CC^{\alpha}(\bar{\Omega}^{\varepsilon}_{T})}].
\]
In summary, we conclude that it is enough to prove Theorem
\ref{t5.2} to problem \eqref{5.14}. To this end, first, we reduce
the analyze of the linear system \eqref{5.14} to the study of the
linear initial-boundary value problem with the dynamic boundary
condition for the function $w$. Namely, in the light of the first
condition on $\Gamma^{\varepsilon}_{ T}$ in \eqref{5.14} and the
properties of the function $A$, we assert that the function $w$
solves the problem
\begin{equation}\label{5.15}
\begin{cases}
\Delta w=0\quad\text{in}\quad \Omega^{\varepsilon}_{T},\\
\dfrac{\partial w}{\partial t}+\dfrac{A(x_{1})}{\gamma}\dfrac{\partial w}{\partial x_{2}}=\dfrac{A(x_{1})}{\gamma}\bar{f}_{1}(x,t)\quad\text{on}\quad
 \Gamma^{\varepsilon}_{ T},\\
\dfrac{\partial w}{\partial x_{1}}=0\quad\text{on}\quad \Gamma^{\varepsilon}_{1,T}\cup\Gamma^{\varepsilon}_{3,T},\\
\dfrac{\partial w}{\partial x_{2}}=0\quad\text{on}\quad \Gamma_{2,T},\\
w(x_{1},x_{2},0)=0\quad\text{in}\quad \bar{\Omega}^{\varepsilon}.
\end{cases}
\end{equation}

Thus, we are left to  prove Theorem \ref{t5.2} to problem \eqref{5.15}. The strategy of this proof is the following: first, we obtain the existence and uniqueness of a weak solutions to \eqref{5.15}. Then, we show that the weak solution is  more regular.

\noindent
$\bullet$ At this point, we define the weak solution of \eqref{5.15} as the function $w$ satisfying regularity
\[
w\in L^{\infty}((0,T);
W_{0}^{1,2}(\Omega^{\varepsilon}\cup\Gamma^{\varepsilon}))\quad\text{and}\quad
\frac{\partial w}{\partial t} \in L^{2}(\Gamma^{\varepsilon}_{
T}),
\]
 and the identity
\[
\int\int\limits_{\Omega^{\varepsilon}_{T}}\nabla w\nabla \Psi dx
dt+\gamma\int\limits_{\Gamma^{\varepsilon}_{ T}}\frac{\partial
w}{\partial t}\frac{\Psi}{A(x_{1})}d\omega
dt=\gamma\int\limits_{\Gamma^{\varepsilon}_{ T}}\bar{f}_{1}\Psi
d\omega dt
\]
for any $\Psi\in L^{2}((0,T);
W_{0}^{1,2}(\Omega^{\varepsilon}\cup\Gamma^{\varepsilon}))$.

Standard arguments and results from \cite[Sec. 1]{G}
provide  both the existence of the weak solution in the sense written above, and the validity of the estimates
\begin{align}\label{5.16}
\|w\|_{L^{\infty}((0,T); W^{1,2}(\Omega^{\varepsilon}))}&\leq C\|\bar{f}_{1}\|_{L^{\infty}((0,T); W^{1,2}(\Gamma^{\varepsilon}))},\\
\notag
\|w\|_{L^{\infty}(\bar{\Omega}^{\varepsilon}_{ T})}&\leq C\|\bar{f}_{1}\|_{L^{\infty}((0,T); W^{1,2}(\Gamma^{\varepsilon}))}
\leq C
\|\bar{f}_{1}\|_{\CC^{1+\alpha}(\Gamma^{\varepsilon}_{T})}.
\end{align}

\noindent
$\bullet$  Coming to the regularity of the weak solution, we apply the Schauder approach. Namely, using the partition of unity together with the local diffeomorphisms, Lemma \ref{l5.3}, the second estimates in \eqref{5.16} and the results of Section 3 in \cite{LU}    arrive at the inequality
\[
\|w\|_{\CC^{2+\alpha}(\bar{\Omega}^{\varepsilon}_{ T})}+\|\partial w/\partial  t\|_{\hat{\CC}^{2+\alpha}(\Gamma^{\varepsilon}_{ T})}\leq
C \|\bar{f}_{1}\|_{\CC^{1+\alpha}(\Gamma^{\varepsilon}_{T})}
 \leq
 C[\|f_{0}\|_{\CC^{\alpha}(\bar{\Omega}^{\varepsilon}_{ T})}+
\|f_{1}\|_{\CC^{1+\alpha}(\Gamma^{\varepsilon}_{ T})}.
\]
Finally, recalling that
$\sigma=\frac{w}{A}\Big|_{\Gamma^{\varepsilon}_{ T}}$ and taking
in advantage of the smoothness and the positivity of $A$ (see
assumptions \textbf{(h7)} and \textbf{(h8)}), we arrive at  $\sigma\in
\hat{\CC}^{2+\alpha}(\Gamma^{\varepsilon}_{ T})$ and the
corresponding estimates hold. That completes the proof of
Theorem \ref{t5.2}.
\end{proof}


\subsection{Solvability of nonlinear problem \eqref{5.10}: proof of Theorem \ref{t5.1}}
\label{s5.6}

\noindent
First we introduce the functional spaces $\mathcal{H}_{1}$ and  $\mathcal{H}_{2}$ such that $\mathbf{z}\in\mathcal{H}_{1}$ and $\mathcal{F}\mathbf{z}\in\mathcal{H}_{2}$:
\begin{align*}
\mathcal{H}_{1}&=\CC_{0}^{2+\alpha}(\bar{\Omega}^{\varepsilon}_{ T})\times\hat{\CC}_{0}^{2+\alpha}(\Gamma^{\varepsilon}_{ T}),\\
\mathcal{H}_{2}&=\CC_{0}^{\alpha}(\bar{\Omega}^{\varepsilon}_{
T})\times\CC_{0}^{2+\alpha}(\Gamma^{\varepsilon}_{
T})\times\CC_{0}^{1+\alpha}(\Gamma^{\varepsilon}_{ T})\times
\CC_{0}^{1+\alpha}(\Gamma^{\varepsilon}_{1, T})\times
\CC_{0}^{1+\alpha}(\Gamma_{2,
T})\times\CC_{0}^{1+\alpha}(\Gamma^{\varepsilon}_{3,T}),
\end{align*}
and
\begin{align*}
\|\mathbf{z}\|_{\mathcal{H}_{1}}&=\|(u,\sigma)\|_{\mathcal{H}_{1}}=\|u\|_{\CC^{2+\alpha}(\bar{\Omega}^{\varepsilon}_{
T})}
+\|\sigma\|_{\hat{\CC}^{2+\alpha}(\Gamma^{\varepsilon}_{T})},\\
\|\mathcal{F}\mathbf{z}\|_{\mathcal{H}_{2}}&=\|(\mathcal{F}_{0}(\mathbf{z}),0,\mathcal{F}_{1}(\mathbf{z}),0,0)\|_{\mathcal{H}_{2}}
=\|\mathcal{F}_{0}(\mathbf{z})\|_{\CC^{\alpha}(\bar{\Omega}^{\varepsilon}_{
T})}
+\|\mathcal{F}_{1}(\mathbf{z})\|_{\CC^{1+\alpha}(\Gamma^{\varepsilon}_{
T})}.
\end{align*}
Taking into account representation \eqref{5.10}, we have
\[
\A \mathbf{z}=\FF(x,t)+\bar{\mathcal{F}}(\mathbf{z}),
\]
where  $\A:\mathcal{H}_{1}\to\mathcal{H}_{2}$ is the linear operator studied in Subsection \ref{s5.5}, the vector $\FF$ is constructed by the initial data, $\bar{\mathcal{F}}(\mathbf{z})$ contains the elements described in Corollary \ref{c5.1}.

After that, the direct calculations (see e.g. Section 5.2 \cite{V3}) and Theorem \ref{t5.2}, Corollary \ref{c5.1} and Lemma \ref{l5.2} arrive at the statement.
\begin{lemma}\label{l5.4}
Let $B_{d},$ $B_{d}\subset\mathcal{H}_{1},$ be a ball with the center located in the origin and the radius of $d$. We assume that conditions of Theorem \ref{t5.1} hold. Then, for $\mathbf{z}\in B_{d}$, the following estimates hold:
\[
\|\mathcal{F}(0)\|_{\mathcal{H}_{2}}\leq C_{1}(T),\qquad
\|\mathcal{F}(\mathbf{z}_{1})-\mathcal{F}(\mathbf{z}_{2})\|_{\mathcal{H}_{2}}\leq C_{2}(T,d)\|\mathbf{z}_{1}-\mathbf{z}_{2}\|_{\mathcal{H}_{1}},
\]
where the quantities $C_{1}(T)$ and $C_{2}(d,T)$ vanish if $T, d\to 0$.
\end{lemma}

Then, due to the operator $\A$ satisfies all the assumptions of Theorem \ref{t5.2},  nonlinear problem \eqref{5.10} can be rewritten as
\[
\mathbf{z}=\A^{-1}\FF(x,t)+\A^{-1}\bar{\mathcal{F}}(\mathbf{z})\equiv\mathcal{T}(\mathbf{z}).
\]
Finally, inequalities in Lemma \ref{l5.4} ensure that for sufficiently small $T$ and $d$ the nonlinear operator $\mathcal{T}(\mathbf{z})$ meets the requirements of the fixed point theorem for a contraction operator. Hence,  the equation  $\mathbf{z}=\mathcal{T}(\mathbf{z})$ has the fixed point, which is obviously
a unique solution of \eqref{5.10}. That completes the proof of Theorem \ref{t5.1}. \qed

Actually, with nonessential modification in the proof of Theorem \ref{t3.1}, the very same results hold for  more general configurations of $\Omega^{\varepsilon}$, namely if $\rho(y_{1},0)\neq 0,$ and the function $\Phi^{\varepsilon}$ independent of $\varepsilon$. Thus,  problem \eqref{1.2} is rewritten as
\begin{equation}\label{5.17}
\begin{cases}
\Delta_{y}p^{\varepsilon}=0\qquad\quad\text{in}\quad\Omega^{\varepsilon}(t),\quad t\in(0,T),\\
p^{\varepsilon}=0\qquad\qquad\text{and}\quad \dfrac{\partial p^{\varepsilon}}{\partial \mathbf{n}_{t}}=-\gamma V_{\mathbf{n}}\quad \text{on}\quad \Gamma^{\varepsilon}(t),\, t\in[0,T],\\
\dfrac{\partial p^{\varepsilon}}{\partial \mathbf{n}}=\Phi^{\varepsilon}(y,t)\quad \text{on}\quad\partial\Omega^{\varepsilon}(t)\backslash\Gamma^{\varepsilon}(t), \, t\in[0,T],
\\
\rho(y_{1},0)=\rho_{0}(y_{1}),\quad y_{1}\in[0,l].
\end{cases}
\end{equation}
First, we introduce the following additional hypotheses.

\noindent \textbf{(h9):}
Let the  function
\begin{equation*}
\Phi^{\varepsilon}(y,t)=
\left\{
  \begin{array}{ll}
    \chi_{2}(y_{2}) \varphi_{1}(y_{2},t), & y_2 \in \Gamma^{\varepsilon}_{1}(t), \ \ t\in[0,T], \\[2mm]
   \chi_{1}(y_{1})\, \varphi_{2}(y_{1},t), &  y_1 \in \Gamma_{2}, \ \ t\in[0,T],\\[2mm]
    \chi_{2}(y_{2})\varphi_{3}(y_{2},t), & y_2 \in \Gamma^{\varepsilon}_{3}(t), \ \ t\in[0,T],
  \end{array}
\right.
\end{equation*}
where  $\chi_{i}\in\C_{0}^{\infty}(\R^{1}),$ $i\in \{1,2\},$ are the cut-off functions, $\chi_{1}$ is defined in \textbf{(h3)} and
$$
\chi_{2}(y_{2})=
\begin{cases}
1,\quad\text{if}\quad y_{2}\in[\frac{2\varepsilon}{5} , \frac{3\varepsilon}{5}],\\[2mm]
0,\quad\text{if}\quad y_{2}\notin( \frac{\varepsilon}{5} , \frac{4\varepsilon}{5}),
\end{cases}
$$
$
\varphi_{1} \in\C([0,T];\C^{2+\alpha}[0,\varepsilon]),$ $
\varphi_{2} \in\C([0,T];\C^{2+\alpha}[0,l]),$ $
\varphi_{3} \in\C([0,T];\C^{2+\alpha}[0,\varepsilon]).
$

\noindent \textbf{(h10):}
We assume that the nonnegative  function $\rho_{0}\in \C^{3+\alpha}([0,l])$ meets requirement
\[
\rho_{0}'(0)=\rho_{0}'(l)=0.\]
Note that the last relations mean that $\Gamma^{\varepsilon}$ forms right angles with $\Gamma_{1}^{\varepsilon}$ and $\Gamma_{3}^{\varepsilon}$.

\begin{theorem}\label{t5.7}
Under assumptions \textbf{(h1)}, \textbf{(h4)}, \textbf{(h9)}, \textbf{(h10)}, the results of Theorem \ref{t3.1} hold for problem \eqref{1.1}, \eqref{5.17}.
\end{theorem}


\section{Proof of Theorem \ref{t4.1}}
\label{s6}

In order to prove Theorem \ref{t4.1} we will use the following approach. First, appealing to technique in \cite{KM,Mel,Mel_Sad}, we obtain formal representation for the solution $p^{\varepsilon}$ and a Cauchy problem for the function~$\S$.
Then in subsection~\ref{s4.2} we justify those constructions by finding the residuals left by the approximation function~$\mathcal{P}^{\varepsilon}$  in problem \eqref{1.1}-\eqref{1.2} and estimate them
using properties of special boundary-layer solutions and Lemma~\ref{l4.5}.


\subsection{Formal asymptotic procedure}\label{s4.1}

Following the approach of \cite{KM,Mel_Sad}, we seek the asymptotics of $p^\varepsilon$ in the form
\begin{equation}\label{anzatz}
p^\varepsilon(y, t) \approx w_0(y_1,t) + \varepsilon \, w_1(y_1,t)
 +  \varepsilon^{2} u_2 \left( y_1, \frac{y_2}{\varepsilon}, t \right)
 +  \varepsilon^{3} u_3 \left( y_1, \frac{y_2}{\varepsilon}, t \right)
\end{equation}
Substituting  representation (\ref{anzatz})  in the relations of \eqref{4.3},
taking into account the view of the function $\Phi^{\varepsilon}$ (see \eqref{Phi}),
and then collecting coefficients at the same power of $\varepsilon$, we conclude that the unknown functions $u_2$ and $u_3$ solve the Neumann problems:
\begin{equation}\label{probl_3}
\begin{cases}
- \dfrac{\partial^{2} u_2}{\partial \xi^2_2 }(y_1, {\xi}_2,t)
  =
\dfrac{\partial^{2} w_0}{\partial y_1^2} (y_1,t),
  \qquad\qquad \xi_2 \in (0, \mathcal{S}(y_1,t)),
\\\\
 \dfrac{\partial u_2}{\partial \xi_2}\big(y_1, \mathcal{S}(y_1,t), t\big)
  =  \dfrac{\partial \mathcal{S}}{\partial y_{1}} \, \dfrac{\partial w_0}{\partial y_{1}} - \gamma \dfrac{\partial \mathcal{S}}{\partial t},
	\\\\
 \dfrac{\partial u_2}{\partial \xi_2}(y_1, 0,t)
  = - \chi_{1}(y_{1})\, \varphi_{2}(y_{1},t),
 \end{cases}
\end{equation}
and
\begin{equation}\label{probl_4}
\begin{cases}
- \dfrac{\partial^{2} u_3}{\partial \xi^2_2 }(y_1, {\xi}_2, t)
 =
\dfrac{\partial^{2} w_1}{\partial y_1^2} (y_1,t),
  \qquad\qquad \xi_2 \in (0, \mathcal{S}(y_1,t)),
\\\\
 \dfrac{\partial u_3}{\partial \xi_2}\big(y_1, \mathcal{S}(y_1,t),t\big)
  =  \dfrac{\partial \mathcal{S}}{\partial y_{1}} \, \dfrac{\partial w_1}{\partial y_{1}},
\\\\
 \dfrac{\partial u_3}{\partial \xi_2}(y_1, 0,t)
  =  0.
 \end{cases}
\end{equation}
Here, the variables $y_1\in [0,l]$ and $t \in [0, T]$ are regarded as  parameters.

Writing down the necessary and sufficient condition for solvability of problem (\ref{probl_3}), we derive the differential equation
\begin{equation}\label{diff_w_0}
\dfrac{\partial}{\partial y_1}\left( \mathcal{S}(y_1,t)  \dfrac{\partial w_0(y_1,t)}{\partial{y_1}}  \right) = \gamma \dfrac{\partial \mathcal{S}}{\partial t}(y_1,t) - \chi_{1}(y_{1})\, \varphi_{2}(y_{1},t), \quad y_1\in (0, l).
\end{equation}
Boundary conditions for \eqref{diff_w_0}, as well the solvability of the corresponding boundary-value problem, will be discussed later.

Let $w_0$ be a solution of \eqref{diff_w_0}. Thus, there exists a solution to problem  \eqref{probl_3} up to an additive value that is a function of  variables $y_1$ and $t.$ Taking the integral condition in problem \eqref{4.3} into account, we can choose this function so  that
\begin{equation}\label{cond_u_2}
  u_2\big(y_1, \mathcal{S}(y_1,t), t\big) = 0, \quad \forall\, y_1 \in (0, l), \ \ \forall \, t \in [0, T],
\end{equation}
which provides a unique choice of the function $u_{2}$.

Recasting the same arguments in the case of problem \eqref{probl_4}, we draw out
\begin{equation}\label{diff_w_1}
\dfrac{\partial}{\partial y_1}\left( \mathcal{S}(y_1,t)  \dfrac{\partial w_1}{\partial{y_1}} (y_1,t) \right) = 0, \quad y_1\in (0, l).
\end{equation}
Since $\Phi^\varepsilon =\mathcal{O}(1)$ (as $\varepsilon \to 0)$  on $\Gamma^{\varepsilon}_{1}(t)$ and $\Gamma^{\varepsilon}_{3}(t),$ the derivative $\frac{\partial w_1}{\partial y_{1}}$ has to be equal zero in the points $y_1=0$ and $y_1 =l.$ This means that the function $w_{1}$ depends  only on $t$, i.e. $w_1=w_{1}(t)$.
 Then, taking into account the integral condition in \eqref{4.3}, we arrive at the identity $w_{1}\equiv 0$.
This equality and the same arguments applied to \eqref{probl_4} result the relation
 $u_3 \equiv 0.$

Thus,  ansatz \eqref{anzatz} is rewritten in the form
\begin{equation*}
 w_0(y_1,t)  +  \varepsilon^{2} u_2 \left( y_1, \dfrac{y_2}{\varepsilon}, t \right).
\end{equation*}
To find boundary conditions for a solution of  differential equation \eqref{diff_w_0}  and to satisfy
the boundary conditions on $\Gamma^{\varepsilon}_{1}(t)$ and $\Gamma^{\varepsilon}_{3}(t),$  we should launch the boundary-layer asymptotics. To this end, in  $\delta$-neighborhoods of $\Gamma^{\varepsilon}_{1}(t)$ and $\Gamma^{\varepsilon}_{3}(t),$ we seek first terms in the form
\begin{equation}\label{prim+}
\varepsilon \, \Pi_1\left(\frac{y_1}{\varepsilon}, \frac{y_2}{\varepsilon}, t\right)
\quad \text{and} \quad \varepsilon \, \Pi^*_1\left(\frac{l - y_1}{\varepsilon},\frac{y_2}{\varepsilon},t\right),
\end{equation}
respectively.
Taking  into account Corollary \ref{c4.1} and substituting $\varepsilon  \Pi_1$  in \eqref{4.3},  we get the boundary value problem
\begin{equation}\label{prim+probl}
\left\{\begin{array}{rcll}
  \Delta_{\xi}\Pi_1(\xi,t) &=& 0, &     \xi = (\xi_1, \xi_2) \in (0,+\infty)\times (0, \mathcal{S}_0(t)),
  \\[3mm]
  \dfrac{\partial\Pi_1(\xi_1,0,t)}{\partial \xi_2}& =&\dfrac{\partial\Pi_1(\xi_1,\mathcal{S}_0(t),t)}{\partial \xi_2}\ = \ 0,      &  \xi_1\in(0,+\infty),
 \\[3mm]
\dfrac{\partial\Pi_1(0,\xi_2,t)}{\partial \xi_1} &=&   \Upsilon_1(\xi_2,t),& \xi_2 \in (0,\mathcal{S}_0(t)),
   \\[3mm]
  \Pi_1(\xi_1,\xi_2,t)& \to&   0, &   \xi_1 \to+\infty,\, \xi_2 \in \textcolor[rgb]{1,0,0}{[}0, \mathcal{S}_0(t)\textcolor[rgb]{1,0,0}{]},
 \end{array}
 \right.
\end{equation}
where  $\xi_1 = \dfrac{y_1}{\varepsilon},$ $\xi_2 = \dfrac{y_2}{\varepsilon},$ \
$
\Upsilon_1(\xi_2,t) = - \chi_{2}( \xi_2)\, \varphi_{1}(\xi_2,t) - \dfrac{\partial w_0}{\partial{y_1}} (0,t).
$

Then the method of separation of variables  allows us to find a solution of problem \eqref{prim+probl} in the form
\begin{equation}\label{view_solution}
\Pi_1(\xi,t)=\sum\limits_{m=0}^{+\infty} a_m(t)\exp\Big(-\tfrac{\pi m \xi_1}{\mathcal{S}_0(t)}\Big) \cos\left(\tfrac{\pi m}{\mathcal{S}_0(t)}\xi_2\right),
\end{equation}
where
\begin{gather*}
a_0(t)=\frac{1}{\mathcal{S}_0(t)}\int\limits_{0}^{\mathcal{S}_0(t)}\Upsilon_1(\xi_2,t) \, d\xi_2 = - \langle\langle \chi_{2}( \cdot)\, \varphi_{1}(\cdot,t)\rangle\rangle_{\mathcal{S}_0} - \frac{\partial w_0}{\partial{y_1}} (0,t)=0,
\\
a_m(t)=\frac{2}{\mathcal{S}_0(t)}\int\limits_{0}^{\mathcal{S}_0(t)}\Upsilon_1(\xi_2,t) \, \cos\left(\frac{\pi m}{\mathcal{S}_0(t)}\xi_2\right)d\xi_2, \quad m\in \mathbb{N}.
\end{gather*}
We remark that the fourth condition in (\ref{prim+probl}) leads to
 the equality for the coefficient $a_0$. In summary,
we end up with the  boundary condition
\begin{equation}\label{bv_left}
\frac{\partial w_0}{\partial{y_1}} (0,t) =
- \langle\langle \chi_{2}( \cdot)\, \varphi_{1}(\cdot,t)\rangle\rangle_{\mathcal{S}_0}.
\end{equation}
Repeating the same arguments leading to relations \eqref{prim+probl} and \eqref{view_solution}, we conclude that the unknown function~ $\Pi^*_1$  has to solve  the  problem
\begin{equation}\label{prim+probl+}
\left\{\begin{array}{rcll}
  \Delta_{\xi^{\star}}\Pi^{\star}_1(\xi^{\star},t) & = &0, &\xi^{\star} = (\xi^{\star}_1, \xi_2) \in(0,+\infty)\times (0, \mathcal{S}_l(t)),
  \\[3mm]
  \dfrac{\partial\Pi^{\star}_1(\xi^{\star}_1,0,t)}{\partial \xi_2} &=&\dfrac{\partial\Pi^{\star}_1(\xi^{\star}_1,\mathcal{S}_l(t),t)}{\partial \xi_2} \ = \  0,    &  \xi^{\star}_1\in(0,+\infty),
 \\[3mm]
\dfrac{\partial\Pi^{\star}_1(0,\xi_2,t)}{\partial \xi^{\star}_1} &=&   \Upsilon^{\star}_1(\xi_2,t), & \xi_2 \in (0,\mathcal{S}_l(t)),
   \\[3mm]
  \Pi^{\star}_1(\xi^{\star}_1,\xi_2,t) &\to &  0, &    \xi^{\star}_1 \to+\infty,\, \xi_2 \in \textcolor[rgb]{1,0,0}{[}0, \mathcal{S}_l(t)\textcolor[rgb]{1,0,0}{]},
\end{array}
 \right.
\end{equation}
where $\xi^\star_1 = \dfrac{l - y_1}{\varepsilon},$ $\xi_2 =\dfrac{y_2}{\varepsilon},$ \
$
\Upsilon^{\star}_1(\xi_2,t) = \chi_{2}( \xi_2)\, \varphi_{3}(\xi_2,t) - \dfrac{\partial w_0}{\partial{y_1}} (l,t).
$
Besides, the solution $\Pi^*_1$ is  given with formula \eqref{view_solution}, where we should change $\mathcal{S}_{0}(t)$ and $\Upsilon_1(\xi_2,t)$   by
 $\mathcal{S}_{l}(t)$ and $\Upsilon^{\star}_1(\xi_2,t),$  respectively. Finally,
 the boundary condition has the form
\begin{equation}\label{bv_right}
\frac{\partial w_0}{\partial{y_1}} (l,t) =
 \langle\langle \chi_{2}( \cdot)\, \varphi_{3}(\cdot,t)\rangle\rangle_{\mathcal{S}_l}.
\end{equation}

\begin{remark}\label{r4.1}
In virtue of \eqref{view_solution} and Corollary \ref{c4.1},     the  asymptotic relations hold
\begin{equation}\label{as_estimates}
\begin{array}{c}
  \Pi_1 =  \mathcal{O}(\exp(- \pi \xi_1)),\quad
	\frac{\partial \Pi_1}{\partial \xi_{1}} =  \mathcal{O}(\exp(- \pi \xi_1)),\quad \frac{\partial \Pi_1}{\partial \xi_{2}} =  \mathcal{O}(\exp(- \pi \xi_1)) \quad \mbox{as} \quad \xi_1\to+\infty,
  \\[2mm]
 \Pi_1^{\star}= \mathcal{ O}(\exp(- \pi \xi^\star_1))
\quad
	\frac{\partial \Pi_1^{\star}}{\partial \xi_{1}^{\star}} =  \mathcal{O}(\exp(- \pi \xi^\star_1)),\quad \frac{\partial \Pi_1^{\star}}{\partial \xi_{2}} =  \mathcal{O}(\exp(- \pi \xi^\star_1))
 \quad \mbox{as} \quad \xi^\star_1\to+\infty
\end{array}
\end{equation}
for all $t\in[0,T]$ and either $\xi_{2}\in[0,\S_{0}(t)]$ in the case of $\Pi_{1}$ or $\xi_{2}\in[0,\S_{l}(t)]$ for $\Pi_{1}^{\star}$.
 \end{remark}

Collecting relations \eqref{diff_w_0}, \eqref{bv_left} and \eqref{bv_right}, we deduce that, for each $t\in[0,T],$   unknown functions $\mathcal{S}$ and $w_{0}$ satisfy the problem
\begin{equation}\label{limit_probl}
\begin{cases}
\dfrac{\partial}{\partial y_1}\left( \mathcal{S}(y_1,t)  \dfrac{\partial w_0(y_1,t)}{\partial{y_1}}  \right) = \gamma \dfrac{\partial \mathcal{S}}{\partial t}(y_1,t) - \chi_{1}(y_{1})\, \varphi_{2}(y_{1},t), \quad y_1\in (0, l),
\\\\
\dfrac{\partial w_0}{\partial{y_1}} (0,t) =
- \langle\langle \chi_{2}( \cdot)\, \varphi_{1}(\cdot,t)\rangle\rangle_{\mathcal{S}_0},
\\\\
\dfrac{\partial w_0}{\partial{y_1}} (l,t) =
 \langle\langle \chi_{2}( \cdot)\, \varphi_{3}(\cdot,t)\rangle\rangle_{\mathcal{S}_l}.
\end{cases}
\end{equation}

In order to satisfy the solvability condition for problem \eqref{limit_probl},  we, first, suppose  that  $\mathcal{S}$ is a solution to the ordinary differential equation
\begin{equation*}
  \gamma \dfrac{\partial \mathcal{S}(y_1,t)}{\partial t} = \chi_{1}(y_{1})\, \varphi_{2}(y_{1},t) +  h_0(t), \quad t\in (0, T),
\end{equation*}
where an unknown  function $h_0$ will be defined below. After that,  problem \eqref{limit_probl} becomes as follows:
\begin{equation}\label{limit_probl+}
\begin{cases}
\dfrac{\partial}{\partial y_1}\left( \mathcal{S}(y_1,t)  \dfrac{\partial w_0(y_1,t)}{\partial{y_1}}  \right) = h_0(t), \quad y_1\in (0, l),
\\\\
\dfrac{\partial w_0}{\partial{y_1}} (0,t) =
- \langle\langle \chi_{2}( \cdot)\, \varphi_{1}(\cdot,t)\rangle\rangle_{\mathcal{S}_0},
\\\\
\dfrac{\partial w_0}{\partial{y_1}} (l,t) =
 \langle\langle \chi_{2}( \cdot)\, \varphi_{3}(\cdot,t)\rangle\rangle_{\mathcal{S}_l}.
\end{cases}
\end{equation}
Writing down the necessary and sufficient condition for  solvability of problem (\ref{limit_probl+}) and taking into account Corollary~\ref{c4.1}, we end up with
\begin{align*}
  h_0(t) &= \frac{1}{l} \left( \int_{0}^{\mathcal{S}_l(t)} \chi_{2}( \xi_2)\, \varphi_{3}(\xi_2,t) \, d\xi_2  +
  \int_{0}^{\mathcal{S}_0(t)} \chi_{2}( \xi_2)\, \varphi_{1}(\xi_2,t) \, d\xi_2\right)
  \\
&=  \frac{1}{l} \left( \int_{0}^{1} \chi_{2}( \xi_2)\, \varphi_{3}(\xi_2,t) \, d\xi_2  +
  \int_{0}^{1} \chi_{2}( \xi_2)\, \varphi_{1}(\xi_2,t) \, d\xi_2\right) , \quad t\in [0, T].
\end{align*}
Finally, taking into advantage of  condition \eqref{4.0},  we come to the  Cauchy problem
\begin{equation}\label{Cauchy}
\begin{cases}
\gamma \dfrac{\partial \mathcal{S}(y_1,t)}{\partial t}
 =  \chi_{1}(y_{1})\, \varphi_{2}(y_{1},t)  + h_0(t), \quad
  t \in (0, T),
\\[3mm]
\mathcal{S}(y_1,0)   =  1,
\end{cases}
\end{equation}
which  has a unique solution for every $y_1\in [0, l].$
\begin{remark}\label{r4.3}
If for all $t \in [0,T]$ and $y_1 \in [0, l]$ the inequality
$$
\chi_{1}(y_{1})\, \varphi_{2}(y_{1},t)  + \frac{1}{l} \left( \int_{0}^{1} \chi_{2}( \xi_2)\, \varphi_{3}(\xi_2,t) \, d\xi_2  +
  \int_{0}^{1} \chi_{2}( \xi_2)\, \varphi_{1}(\xi_2,t) \, d\xi_2\right) > 0
$$
holds, then $\frac{\partial \mathcal{S}(y_1,t)}{\partial t}>0,$   and, as a consequence, domain $\Omega_\varepsilon(t)$ increases in time for $\varepsilon$ small enough. Moreover, this condition ensures the fulfillment of  assumption \eqref{3.4}.
\end{remark}

Thus,  Neumann problem \eqref{limit_probl+} has a classical solution up to a function $\eta_0$  depending on $t\in[0,T].$
We can choose it in a such way to fulfill the integral condition in \eqref {4.3}. As a result, we get the solution
$$
\mathfrak{w}_0(y_1, t) = w_0(y_1, t)  -  \frac{1}{|\Gamma^\varepsilon(t)|} \int_{\Gamma^\varepsilon(t)} w_0 \, d\ell, \quad y_1\in [0, l], \ \ t \in [0,T],
$$
that satisfies the equality
\begin{equation}\label{t1}
\int_{\Gamma^{\varepsilon}(t)} \mathfrak{w}_0 \, d\ell = 0, \quad \forall \, t\in[0,T].
\end{equation}



\subsection{Justification}\label{s4.2}

First,  we determine the unique solution $\mathcal{S}$  to problem \eqref{Cauchy}, which is given by formula \eqref{S}. Here, we essential use the  assumptions \textbf{(h3)} and \textbf{(h5)} that provides the following smoothness of the function $\S$:
\begin{equation}\label{4.4}
\S\in\C([0,T],\C^{3}([0,l])),\quad\frac{\partial\S}{\partial t}\in \C([0,T],\C^{3}([0,l])).
\end{equation}
 In the next step, we obtain the unique smooth solution $\mathfrak{w}_0$ to problem
  \eqref{limit_probl}, which satisfies \eqref{t1}.
  After that, returning to problem \eqref{probl_3}, we conclude the uniqueness  and smoothness of its solution $u_2$ that satisfies  condition \eqref{cond_u_2}. Finally, there exist solutions to problems \eqref{prim+probl} and
\eqref{prim+probl+} with  asymptotics~\eqref{as_estimates}.

At this point, we start to estimate the difference  between the classical solution $p^{\varepsilon}$
and the approximation function
\begin{equation*}
  \mathcal{P}^\varepsilon(y, t) := \mathfrak{w}_0(y_1,t)  +  \varepsilon^{2} u_2 \left( y_1, \dfrac{y_2}{\varepsilon}, t \right), \quad y \in \Omega^\varepsilon(t),\quad t\in[0,T],
  \label{approx}
\end{equation*}
in the norm of the space $\C\big([0,T]; \, H^{1}(\Omega^{\varepsilon}(t))\big)$.

Substituting $\mathcal{P}^\varepsilon$ in the differential equation and the boundary conditions of problem~(\ref{4.3}) and taking into account  relations in problems   \eqref{limit_probl}, \eqref{probl_3},  Corollary~\ref{c4.1} and \eqref{cond_u_2}, we find that $  \mathcal{P}^\varepsilon$ solves the problem
\begin{equation*}
\begin{cases}
\Delta_{y}\mathcal{P}^\varepsilon  =  R^\varepsilon_1 \quad \text{in} \quad \Omega^\varepsilon(t),
\\\\
\dfrac{\partial \mathcal{P}^\varepsilon}{\partial y_{2}}=\varepsilon
\dfrac{\partial \mathcal{S}}{\partial y_{1}}  \dfrac{\partial \mathcal{P}^\varepsilon}{\partial
y_{1}}-\varepsilon \gamma \dfrac{\partial \mathcal{S}}{\partial t} + R^\varepsilon_2 \quad
\text{on}\ \Gamma^{\varepsilon}(t),
\\\\
\int_{\Gamma^{\varepsilon}(t)} \mathcal{P}^\varepsilon \, d\ell = 0,
\\\\
- \dfrac{\partial \mathcal{P}^\varepsilon}{\partial y_1}=
\langle\langle \chi_{2}( \cdot)\, \varphi_{1}(\cdot,t)\rangle\rangle_{\mathcal{S}_0}
\quad\text{on}\, \Gamma^{\varepsilon}_1(t),
\\\\
 \dfrac{\partial \mathcal{P}^\varepsilon}{\partial y_1}=
 \langle\langle \chi_{2}( \cdot)\, \varphi_{3}(\cdot,t)\rangle\rangle_{\mathcal{S}_l}
\quad\text{on}\, \Gamma^{\varepsilon}_3(t),
\\\\
\displaystyle{ - \frac{\partial \mathcal{P}^\varepsilon}{\partial y_2}= \varepsilon\chi_{1}( y_{1})\, \varphi_{2}(y_{1},t)}\quad\text{on}\,  \Gamma_2
\end{cases}
\end{equation*}
for any fixed $t\in [0,T].$ Here
$$
R^\varepsilon_1(y,t)= \varepsilon^2 \dfrac{\partial^2 u_2}{\partial y^2_{1}}\left( y_1, \frac{y_2}{\varepsilon}, t \right), \quad
R^\varepsilon_2(y,t)= \varepsilon^3 \, \dfrac{\partial \mathcal{S}}{\partial y_{1}} \,
\dfrac{\partial u_2}{\partial y_{1}}\left( y_1, \frac{y_2}{\varepsilon}, t \right).
$$
Due to condition \textbf{(h3)} and \textbf{(h4)} and \eqref{4.4} we have
\begin{equation}\label{est_1}
  \sup\limits_{t\in[0,T]} \, \sup\limits_{y\in\Omega^{\varepsilon}(t)}
      \left|R^{\varepsilon}_1(y,t)\right|
 \leq C  \varepsilon^{2}, \qquad
  \sup\limits_{t\in[0,T]} \, \sup\limits_{y\in\Gamma^{\varepsilon}(t)}
      \left|R^{\varepsilon}_2(y,t)\right|
 \leq C \, \varepsilon^{3}.
\end{equation}

\begin{remark}
In \eqref{est_1} and further,  all constants in  inequalities are independent
of the functions $\S,$ $\mathcal{P}^\varepsilon,$ $p^\varepsilon,$ the variables $y,$ $t$ and the
parameter $\varepsilon.$
 \end{remark}

As a consequence,  the difference  $W^\varepsilon = p^\varepsilon - \mathcal{P}^\varepsilon$ satisfies the relations
\begin{equation*}
\begin{cases}
-\Delta_{y}W^\varepsilon  =  R^\varepsilon_1 \quad \text{in} \ \Omega^\varepsilon(t),
\\\\
- \dfrac{\partial W^\varepsilon}{\partial {\bf n}_t} = \dfrac{R^\varepsilon_2}{|\nabla_y \mathcal{R}|}  \quad
\text{on}\ \Gamma^{\varepsilon}(t),
\\\\
\displaystyle{\int_{\Gamma^{\varepsilon}(t)} W_\varepsilon \, d\ell} = 0,
\\\\
- \dfrac{\partial W^\varepsilon}{\partial y_1}= \chi_{2}(y_{2}/\varepsilon,t) \varphi_{1}(y_{2}/\varepsilon,t)-
\langle\langle \chi_{2}( \cdot)\, \varphi_{1}(\cdot,t)\rangle\rangle_{\mathcal{S}_0}
\quad\text{on}\ \Gamma^{\varepsilon}_1(t),
\\\\
 \dfrac{\partial W^\varepsilon}{\partial y_1}= \chi_{2}(y_{2}/\varepsilon,t)\varphi_{3}(y_{2}/\varepsilon,t) -
 \langle\langle \chi_{2}( \cdot)\, \varphi_{3}(\cdot,t)\rangle\rangle_{\mathcal{S}_l}
\quad\text{on}\ \Gamma^{\varepsilon}_3(t),
\\\\
 - \dfrac{\partial W^\varepsilon}{\partial y_2}= 0 \quad\text{on}\  \Gamma_2.
\end{cases}
\end{equation*}
We multiply the differential equation in this relations by  $W^\varepsilon$ and, then,  integrate over $\Omega^{\varepsilon}(t)$ for each fixed $t\in[0,T]$. Standard calculations lead to  the equality
 \begin{multline}\label{int_nevyazka}
    \int_{\Omega^\varepsilon(t)} {|\nabla_{y} W^\varepsilon |}^2 \, dy
 =  \int_{\Omega^\varepsilon(t)} R^\varepsilon_1 \, W^\varepsilon \,dy
 -      \int_{\Gamma^\varepsilon(t)}     \frac{R^\varepsilon_2}{|\nabla_y \mathcal{R}|} \, W^\varepsilon \, d\ell
 \\
 + \int_{\Gamma^\varepsilon_1(t)} \Big( \Phi^\varepsilon -
\langle\langle \chi_{2}( \cdot)\, \varphi_{1}(\cdot,t)\rangle\rangle_{\mathcal{S}_0} \Big) \, W^\varepsilon \, dy_2
 + \int_{\Gamma^\varepsilon_3(t)} \Big(\Phi^\varepsilon -
\langle\langle \chi_{2}( \cdot)\, \varphi_{3}(\cdot,t)\rangle\rangle_{\mathcal{S}_l} \Big) \, W^\varepsilon \, dy_2.
\end{multline}
Then, we evaluate  each term in the right-hand side of \eqref{int_nevyazka}.

\noindent $\bullet$ As for  first two terms, appealing to inequalities \eqref{est_1} and  inequalities \eqref{Puankare} and \eqref{l_4}, we deduce
\begin{align}\label{est_2}
 \left|\int_{\Omega^\varepsilon(t)} R^\varepsilon_1 \, W^\varepsilon \,dy\right| &\le C \, \varepsilon^\frac{5}{2} \|W^\varepsilon\|_{L^2(\Omega^\varepsilon(t))} \le C \, \varepsilon^\frac{5}{2} \|\nabla_{y}W^\varepsilon\|_{L^2(\Omega^\varepsilon(t))},
 \\ \label{est_2+}
 \left|\int_{\Gamma^\varepsilon(t)}   \frac{R^\varepsilon_2}{|\nabla_y \mathcal{R}|} \, W^\varepsilon \, d\ell \right| &\le C
   \, \varepsilon^3 \|W^\varepsilon\|_{L^2(\Gamma^\varepsilon(t))}
	\leq C \, \varepsilon^\frac{5}{2} \|\nabla_{y}W^\varepsilon\|_{L^2(\Omega^\varepsilon(t))}.
\end{align}

\noindent $\bullet$
Coming to the last two terms in \eqref{int_nevyazka}, we introduce two smooth cut-off functions
\begin{equation*}
\chi_\delta(y_1)=
\left\{\begin{array}{ll}
1, & \text{if} \ \ y_1 \le \frac{\delta}{2},
\\[1mm]
0, & \text{if} \ \ y_1 \ge  \delta,
\end{array}\right.
\qquad
\chi^\star_\delta(y_1):= \chi_\delta(l - y_1),
\end{equation*}
where $\delta$ is defined in Corollary~\ref{c4.1}, and consider the functions
$$
  \varepsilon \, \chi_\delta(y_1)\,
   \Pi_1\left(\frac{y_1}{\varepsilon}, \frac{y_2}{\varepsilon}, t\right) \quad \text{and} \quad
 \varepsilon \, \chi^\star_\delta(y_1)\,  \Pi^\star_1\left(\frac{l - y_1}{\varepsilon},\frac{y_2}{\varepsilon},t\right).
$$
Taking into account relations in problem \eqref{prim+probl}, the direct calculations entail
\begin{equation*}
\begin{cases}
-\Delta_{y}\left( \varepsilon \chi_\delta  \Pi_1 \right) =  R^\varepsilon_3 \quad \text{in} \ \Omega^\varepsilon(t),
\\\\
- \dfrac{\partial }{\partial {\bf n}}\left( \varepsilon \chi_\delta  \Pi_1 \right) = 0 \quad
\text{on}\ \partial\Omega^\varepsilon(t) \setminus \Gamma^{\varepsilon}_1(t),
\\\\
- \dfrac{\partial}{\partial y_1}\left( \varepsilon \chi_\delta  \Pi_1 \right) = \Phi^\varepsilon -
\langle\langle \chi_{2}( \cdot)\, \varphi_{1}(\cdot,t)\rangle\rangle_{\mathcal{S}_0}
\quad\text{on}\ \Gamma^{\varepsilon}_1(t),
\end{cases}
\end{equation*}
whence
\begin{equation}\label{int_Pi}
 \int_{\Gamma^\varepsilon_1(t)} \Big(\Phi^\varepsilon -
\langle\langle \chi_{2}( \cdot)\, \varphi_{1}(\cdot,t)\rangle\rangle_{\mathcal{S}_0} \Big) \, W^\varepsilon \, dy_2 =-
\int_{\Omega^\varepsilon(t)} R^\varepsilon_3 \, W^\varepsilon \,dy +\varepsilon \int_{\Omega^\varepsilon(t)} \nabla_y\left(\chi_\delta  \Pi_1 \right) \cdot\nabla_y W^\varepsilon\, dy.
\end{equation}
Here
$$
R^\varepsilon_3(y,t) = - 2\frac{d\chi_\delta(y_1)}{dy_1}  \,
\frac{\partial\Pi_1(\xi,t)}{\partial \xi_1}\bigg|_{\xi=\frac{y}{\varepsilon}} -
\varepsilon \, \frac{d^2\chi_\delta(y_1)}{dy_1^2} \,
   \Pi_1(\xi,t)\big|_{\xi=\frac{y}{\varepsilon}}.
$$

Since the function $\Pi_1$  and its derivatives $\frac{\partial\Pi_{1}}{\partial\xi_{i}},$ $i=1,2,$ decrease exponentially  (see  Remark~\ref{r4.1}) and  the support of the derivatives of the cut-off function $\chi_\delta$ belongs to the segment $[\frac{\delta}{2}, \delta],$  we arrive at the inequality
\begin{equation}\label{est_3}
  \sup\limits_{t\in[0,T]} \, \sup\limits_{y\in\Omega^{\varepsilon}(t)}
      \left|R^{\varepsilon}_3(y,t)\right|
 \leq C \, \exp{ \left(-\frac{\pi \delta}{2\varepsilon}\right)}.
\end{equation}
With the help of \eqref{est_3} and the Poincar\'e inequality \eqref{Puankare} in Lemma \ref{l4.5} we derive
\begin{multline}\label{est_4}
\left|  \int_{\Gamma^\varepsilon_1(t)} \Big(\Phi^\varepsilon -
\langle\langle \chi_{2}( \cdot)\, \varphi_{1}(\cdot,t)\rangle\rangle_{\mathcal{S}_0} \Big) \, W^\varepsilon \, dy_2\right| \le C
\sqrt{\varepsilon}\, \exp{ \left(-\frac{\pi \delta}{2\varepsilon}\right)} \Big(\|W^\varepsilon\|_{L^2(\Omega^\varepsilon(t))} +
 \|\nabla W^\varepsilon\|_{L^2(\Omega^\varepsilon(t))}\Big)
\\
+ C \, \varepsilon \, \|\nabla W^\varepsilon\|_{L^2(\Omega^\varepsilon(t))}
\Big(\int_{0}^{+\infty}\int_{0}^{\mathcal{S}_0(t)} |\nabla_\xi \Pi_1(\xi,t)|^2\, d\xi_{2}d\xi_{1} \Big)^{\frac{1}{2}}
\le C \, \varepsilon \, \|\nabla W^\varepsilon\|_{L^2(\Omega^\varepsilon(t))}.
\end{multline}

Similarly arguments and  properties of the solution $\Pi^\star_1$ (see \eqref{prim+probl+}) yield
\begin{equation}\label{est_5}
\left|  \int_{\Gamma^\varepsilon_3(t)} \Big(\Phi^\varepsilon -
\langle\langle \chi_{2}( \cdot)\, \varphi_{3}(\cdot,t)\rangle\rangle_{\mathcal{S}_0} \Big) \, W^\varepsilon \, dy_2\right|
\le
C \, \varepsilon \, \|\nabla W^\varepsilon\|_{L^2(\Omega^\varepsilon(t))}.
\end{equation}

In conclusion, from \eqref{int_nevyazka} in virtue of  \eqref{est_2}, \eqref{est_2+}, \eqref{est_4} and \eqref{est_5}, it follows
the inequality
\begin{equation}\label{ineq_1}
\sup_{t\in[0,T]} \|\nabla W^\varepsilon\|_{L^2(\Omega^\varepsilon(t))} \le C_0\, \varepsilon,
\end{equation}
that together with Lemma \ref{l4.5} complete the proof of Theorem \ref{t4.1}.\qed

Using the Cauchy-Bunyakovsky-Schwarz inequality and \eqref{main_estimate},
we derive the statement.
\begin{corollary}\label{corol_7_4}  For the difference between the solution to problem \eqref{1.1}-\eqref{1.2} and the solution to the limit problem \eqref{limit_prob} the following estimate
$$
\left\| \, \langle\langle p^\varepsilon \rangle\rangle_{\varepsilon \mathcal{S}} - \mathfrak{w}_0 \, \right\|_{\C([0,T]; \, L^{2}(0, l))}\leq
C_0 \, \sqrt{\varepsilon}
$$
holds.
\end{corollary}


\section{Conclusions}
\label{sC}

\noindent
In this work, we discuss the one-phase contact Hele-Shaw problem \eqref{1.1}-\eqref{1.2} with ZST in the domain $\Omega^{\varepsilon}(t)$ that depends on a small parameter $\varepsilon.$  In particular, we analyze the classical local solvability of this problem for each fixed $\varepsilon$ and describe the asymptotic behavior of the solution $p^{\varepsilon}$ as $\varepsilon\to 0$.

As it follows from our consideration, the asymptotic analysis  turns out the effective tools to study of the Hele-Shaw problem in thin domains. Namely, it allows us to obtain not only the explicit representation of the free boundary $\Gamma^{\varepsilon}(t)$ but also to establish  preserving the geometry of the moving boundary in  $\delta$-neighborhoods of the corner points for $t\in[0,T].$  This property is not exactly the waiting time phenomena described in \cite{KLV}, since  the corner points on the free boundary shift  instantly  for $t>0$. However, in opposite to all the early obtained results concerning with the waiting time phenomena,  we can find the size of those $\delta$-neighborhoods that depends on the support of the function $\Phi^{\varepsilon}|_{y_{2}=0}$.

An important task of existing multiscale methods is their stability and accuracy.
The proof of the error estimate between the constructed approximation and the exact solution is a general principle that should be applied to the analysis of the effectiveness of the proposed multiscale method.
In our paper, we have constructed and justified the asymptotic approximation for the solution to problem \eqref{1.1}-\eqref{1.2} and proved the corresponding estimates. The results obtained in Theorem~\ref{t4.1} and  Corollary~\ref{corol_7_4}  argue that
the complex Hele-Shaw problem \eqref{1.1}-\eqref{1.2} can be replaced by the corresponding
Cauchy problem~\eqref{Cauchy} and one-dimensional  limit problem \eqref{limit_prob} with sufficient accuracy measured by the parameter $\varepsilon$ characterizing the thickness of the domain $\Omega^{\varepsilon}(t)$ and the amplitude of the free boundary.

Our ideas can be exported to cover the analysis of problems like \eqref{1.1}-\eqref{1.2} in more general cases. First, our consideration can be extended to the Hele-Shaw problem with nonzero surface tension (NZST), and to the Stefan problem in the case of  both ZST and  NZST.  The proposed approach can be adapted and generalized in order to consider problem like \eqref{1.1}-\eqref{1.2} in three-dimensional case, i.e. $Q\in\R^{3},$ $Q=(0,l_{1})\times(0,l_{2})\times(0,2\varepsilon)$.
Also, it will be very interesting to study Hele-Shaw and Stefan problems in thin domains when a free boundary has
a highly small amplitude, for instance, $\rho = \mathcal{O}(\varepsilon^\alpha)$ as $\varepsilon \to 0$ and $\alpha >1.$
Perhaps all of  this will be the subject of future research.


\section*{Appendix}
\label{sA}

\theoremstyle{definition}
\newtheorem{reAPP}{Remark}[section]
\newtheorem{leAPP}{Lemma}[section]
\renewcommand{\thereAPP}{A.\arabic{reAPP}}
\renewcommand{\theleAPP}{A.\arabic{leAPP}}
\setcounter{equation}{0}
\setcounter{subsection}{0}
\renewcommand{\theequation}{A.\arabic{equation}}
\renewcommand{\thesubsection}{A.\arabic{subsection}}

\subsection{Statement of Lemma \ref{l2.1}}\label{s7.1}

 Denoting the inverse Laplace transformation
with respect to time $t$ by $\mathcal{L}^{-1}_{t}$,
 we recall some properties of the function $K=K(x,t):\R\times[0,T]\to\R:$
\begin{equation}\label{2.3}
K(x,t)=\mathcal{L}_{t}^{-1}\Big(\int_{-\infty}^{+\infty}\frac{e^{i\lambda x}}{p+\mathfrak{C}_{0}|\lambda|}d\lambda\Big)
\end{equation}
with the positive number $\mathfrak{C}_{0}$ and $Re\, p>0$, which are
obtained in Lemma 3.1 \cite{BV1} (where $\mathfrak{C}_{0}=A_{1}$ and
$A_{2}=0$).
\begin{leAPP}\label{l2.1}
Let $\alpha\in(0,1),$ $T>0$ be arbitrary fixed and let $k$ be nonnegative integer. Then for each  $t\in[0,T]$ and $x,x_{1}, x_{2}\in \R$, the following estimates hold:
\begin{itemize}
\item[(i)] $$K(x,t)=\frac{2\mathfrak{C}_{0}t}{(\mathfrak{C}_{0}t)^{2}+x^{2}};$$
\item[(ii)]
$$\int_{0}^{t}d\tau\int_{-\infty}^{+\infty}\frac{\partial^{k}K}{\partial y^{k}}(y,\tau)dy=
\begin{cases}
2\pi t\quad\text{if}\quad k=0,\\
0 \quad\text{if}\quad k>0;
\end{cases}$$
\item[(iii)]
\begin{align*}
\int_{0}^{t}d\tau\int_{0}^{+\infty}|y|^{\alpha}\Big|\frac{\partial K}{\partial y}(y,\tau)\Big|dy&\leq Ct^{\alpha},\\
\int_{0}^{t}d\tau\int_{|y|\leq 2|x_{1}-x_{2}|}|y|^{\alpha}\Big|\frac{\partial K}{\partial y}(y,\tau)\Big|dy&\leq C|x_{1}-x_{2}|^{\alpha},\\
\int_{0}^{t}d\tau\int_{|y|\geq 2|x_{1}-x_{2}|}|y|^{\alpha}\Big|\frac{\partial^{2} K}{\partial y^{2}}(y,\tau)\Big|dy&\leq C|x_{1}-x_{2}|^{\alpha-1}.
\end{align*}
\end{itemize}
\end{leAPP}



\end{document}